\newcommand{\R}{\mathds R}
\newcommand{\C}{\mathds C}
\newcommand{\Z}{\mathds Z}
\newcommand{\N}{\mathds N}

\newcommand{\Ddt}{\tfrac{\mathrm D}{\mathrm dt}}
\newcommand{\DDdt}{\tfrac{\mathrm D^2}{\mathrm dt^2}}

\newcommand{\spfl}{\mathfrak{sf}}
\newcommand{\iMaslov}{\mathrm i_{\mathrm{Maslov}}}

\newcommand{\Dim}{\mathrm{dim}}
\newcommand{\Ker}{\mathrm{Ker}}

\documentclass[twoside,final,a4paper,10pt]{amsart}

\usepackage{times}
\usepackage{dsfont}
\usepackage[all]{xy}

\renewcommand{\contentsline}[3]{\csname new#1\endcsname{#2}{#3}}
\newcommand{\newchapter}[2]{\bigskip\hbox to \hsize{\vbox{\advance\hsize by -.5cm\baselineskip=12pt\parfillskip=0pt\leftskip=2cm\noindent\hskip -2cm #1\leaders\hbox{.}\hfil\hfil\par}$\,$#2\hfil}}
\newcommand{\newsection}[2]{\medskip\hbox to \hsize{\vbox{\advance\hsize by -.5cm\baselineskip=12pt\parfillskip=0pt\leftskip=2.5cm\noindent\hskip -2cm #1\leaders\hbox{.}\hfil\hfil\par}$\,$#2\hfil}}
\newcommand{\newsubsection}[2]{\medskip\hbox to \hsize{\vbox{\advance\hsize by -.5cm\baselineskip=12pt\parfillskip=0pt\leftskip=3.5cm\noindent\hskip -2cm #1\leaders\hbox{.}\hfil\hfil\par}$\,$#2\hfil}}

\numberwithin{equation}{section}

\title[Spectral flow and iteration of closed  geodesics]{Spectral flow and iteration\\ of closed semi-Riemannian geodesics}
\author[M. A. Javaloyes]{Miguel Angel Javaloyes}
\address{Departamento de Matem\'atica,\hfill\break\indent
Universidade de S\~ao Paulo, \hfill\break\indent Rua do Mat\~ao
1010,\hfill\break\indent CEP 05508-900, S\~ao Paulo, SP, Brazil}
\email{majava@ime.usp.br}

\author[P.\ Piccione]{Paolo Piccione}
\address{Departamento de Matem\'atica,\hfill\break\indent
Universidade de S\~ao Paulo, \hfill\break\indent Rua do Mat\~ao
1010,\hfill\break\indent CEP 05508-900, S\~ao Paulo, SP, Brazil}
\email{piccione@ime.usp.br}
\urladdr{http://www.ime.usp.br/\~{}piccione}
\thanks{M. A. J. is sponsored by Fapesp;  P. P. is partially sponsored by CNPq.}

\date{October 16th, 2007}

\begin{document}


\theoremstyle{plain}\newtheorem*{teon}{Theorem}
\theoremstyle{definition}\newtheorem*{defin*}{Definition}
\theoremstyle{plain}\newtheorem{teo}{Theorem}[section]
\theoremstyle{plain}\newtheorem{prop}[teo]{Proposition}
\theoremstyle{plain}\newtheorem{lem}[teo]{Lemma}
\theoremstyle{plain}\newtheorem{cor}[teo]{Corollary}
\theoremstyle{definition}\newtheorem{defin}[teo]{Definition}
\theoremstyle{remark}\newtheorem{rem}[teo]{Remark}
\theoremstyle{plain} \newtheorem{assum}[teo]{Assumption}
\swapnumbers
\theoremstyle{definition}\newtheorem{example}{Example}[section]
\theoremstyle{plain} \newtheorem*{acknowledgement}{Acknowledgements}
\theoremstyle{definition}\newtheorem*{notation}{Notation}


\begin{abstract}
We introduce the notion of spectral flow along a periodic semi-Riemannian geodesic, as a suitable substitute of
the Morse index in the Riemannian case. We study the growth of the spectral flow along a closed geodesic under
iteration, determining its asymptotic behavior.
\end{abstract}

\maketitle

\begin{section}{Introduction}
Closed geodesics are critical points of the geodesic action functional in the
free loop space of a semi-Riemannian manifold $(M,g)$; a very classical problem
in Geometry is to establish multiplicity of closed geodesics (see \cite{Kli}).
By ``multiplicity of closed geodesics'', it is always meant multiplicity of ``prime closed geodesics'', i.e.,
those geodesics that are not obtained by iteration of another closed geodesic.
One of the difficult aspects in the variational theory of geodesics is precisely
the question of distinguishing iterates.
In a celebrated paper by R. Bott \cite{Bo4} it is studied the Morse index of a closed
Riemannian geodesic; the main result is a formula establishing the growth of the
index under iteration. This formula shows that, given a closed geodesic $\gamma$,
its iterates $\gamma^{(N)}$ either have Morse index that grows linearly with $N$,
or they have all null index.  This has been used by Gromoll and Meyer in another celebrated
paper \cite{GroMey2}, where the authors develop an equivariant Morse theory to
prove the existence of infinitely many prime closed geodesics in compact Riemannian
manifolds whose free loop space has unbounded Betti numbers. Roughly speaking,
the (uniform) linear growth of the Morse index of an iterate implies that
if there were only a finite number of prime closed geodesics, then the homology
generated by their iterates would not suffice to produce the homology of the entire
free loop space. Refinements of this kind of results have appeared in subsequent literature
(see \cite{BanKli, Kli}).

More recently, an increasing interest has arised around the question of existence of
periodic solutions of more general variational problems, and especially in the
context of semi-Riemannian geometry. Recall that a semi-Riemannian manifold is
a manifold $M$ endowed with a nondegenerate, but possibly non positive definite,
metric tensor $g$. In this context, the geodesic variational theory is extremely
more involved, even in the fixed endpoint case (see \cite{ABFM, AbbMej}), due
to the strongly indefinite character of the action functional.
When the metric tensor is Lorentzian, i.e., it has index equal to $1$, and the
metric is stationary, i.e., time invariant, then it is possible to perform a certain
reduction of the geodesic variational problem that yields existence results
similar to the positive definite case (see \cite{BilMerPic, CanFloSan, asian, CAG1, Masiello}).
For instance, it is proven in \cite{BilMerPic} that any
stationary Lorentzian manifold having a compact Cauchy surface and whose free loop space
has unbounded Betti numbers has infinitely many distinct prime closed geodesics.
A Bott type result on the Morse index of an iterate has been proven in \cite{JavLopPic}
for stationary Lorentzian metrics.

Dropping the stationarity assumption is at this stage a quite challenging task.
The first problem that one encounters is the fact that the critical points
of the geodesic action functional have \emph{truly} infinite Morse index,
so that standard Morse theory fails. In order to develop Morse theory for strongly
indefinite functionals (see \cite{AbbMej2}), one computes the dimension of the
intersection between the stable and the unstable manifolds as the difference of
a sort of generalized index function defined at each critical point.
In the fixed endpoint geodesic case, such generalized index function can be described explicitly
as a kind of algebraic count of the degeneracies of the index form along
the geodesic. More precisely, this is the so-called \emph{spectral flow} of the path
of index forms along the geodesic. Several extensions of the Morse index theorem (see \cite{asian, topology})
show that this number is related to a symplectic invariant associated to a fixed endpoint
geodesic, called the Maslov index. The Maslov index  is the natural substitute for
the number of conjugate points along a geodesic, which may be infinite when the metric is
non positive definite.

As to the periodic case, the notion of spectral flow $\spfl(\gamma)$ of a closed geodesic $\gamma$
has been introduced only recently (see \cite{BenPic}); this is a generalization of the Morse index of the geodesic
action functional in the Riemannian case. For the reader's convenience,
in Section~\ref{sec:spectralflow} we will review briefly this definition, that is
given in terms of the choice of a periodic frame along the geodesic. An explicit computation
shows that the periodic spectral flow equals the fixed endpoint spectral flow plus a \emph{concavity index},
as in the original paper by M. Morse \cite{Morseconc}, plus a certain degeneracy term (see Theorem~\ref{Morseperiodic}).

The main purpose of the present paper is to establish the growth of the spectral flow under iteration
of the closed geodesic, along the lines of \cite{Bo4}. Given a closed semi-Riemannian geodesic
$\gamma$, we will show the existence of a function $\lambda_\gamma$ defined on the unit circle
$\mathds S^1$ and taking values in $\Z$ (Definition~\ref{thm:defspflofunction}), with the
property that the spectral flow of the
$N$-th iterate $\gamma^{(N)}$ of $\gamma$ equals the sum of the values that $\lambda_\gamma$ takes at the
$N$-th roots of unity (Theorem~\ref{thm:fourier}). This function is continuous, i.e., locally constant, at the points of $\mathds S^1\setminus\{1\}$
that are not eigenvalues of the linearized Poincar\'e map $\mathfrak P_\gamma$ of $\gamma$ (Proposition~\ref{thm:lambdagammaconst});
the jump of $\lambda_\gamma$ at an eigenvalue of $\mathfrak P_\gamma$ is bounded by the
dimension of the corresponding eigenspace (Corollary~\ref{thm:stimasalti}).
As in the Riemannian case, knowing the exact value of the jumps of $\lambda_\gamma$ at each discontinuity
point would determine entirely the function $\lambda_\gamma$. It should be observed that these discontinuities
correspond to isolated degeneracy instants of an \emph{analytic} path of self-adjoint Fredholm operators, and the
value of the jump equals the contribution of the degeneracy instant to the spectral flow of the path.
In principle, these jumps can be computed using higher order methods (see \cite{GiaPicPorCOMPTES}), involving
a finite number of derivatives of the path.
As to the point $z=1$, there is always a discontinuity
of $\lambda_\gamma$ when $g$ is not positive definite (see Corollary~\ref{thm:saltoinz=1});
the value of the jump at $z=1$ equals the index of the metric tensor $g$. Concerning the \emph{nullity}
of the iterates $\gamma^{(N)}$, the semi-Riemannian case is totally analogous to the Riemannian case, where the question is reduced
to studying the spectrum of the linearized Poincar\'e map.

Using these properties of the spectral flow function $\lambda_\gamma$, we then study the asymptotic behavior
of the sequence $N\mapsto\spfl(\gamma^{(N)})\in\Z$, by first showing that the limit $\lim\limits_{N\to\infty}\frac1N\spfl(\gamma^{(N)})$
exists and is finite (Proposition~\ref{thm:esistenzalimitelambda}). More precisely, using a certain finite dimensional reduction,
we show that $\spfl(\gamma^{(N)})$ is the sum of a linear term in $N$, a uniformly bounded term, and the term of a sequence
which is either bounded or it satisfies a sort of uniform linear growth in $N$ (Proposition~\ref{thm:spflredfindim}, Lemma~\ref{thm:unifboundeterm} and
Proposition~\ref{thm:cresictaBN}).   When $\gamma$ is a hyperbolic geodesic, i.e., when $\mathfrak P_\gamma$ has no eigenvalue on
the unit circle, then $\vert\spfl(\gamma^{(N)})\vert$ either grows linearly with $N$, or it is constant equal to the
index of the metric tensor $g$ (Proposition~\ref{thm:gammahyper}). In view to the development of a full-fledged
Morse theory for semi-Riemannian closed geodesics, the most important result is that the spectral flow of an iterate $\gamma^{(N)}$
is either bounded, or it has a uniform linear growth (Proposition~\ref{thm:lineargrowth}, Corollary~\ref{thm:boundedorlinear}).
This implies, in particular, that if a semi-Riemannian manifold has only a finite number of distinct prime
closed geodesics, then for $k\in\Z$ with $\vert k\vert$ sufficiently large, the total number of \emph{geometrically distinct}
closed geodesics whose spectral flow is equal to $k$ has to be uniformly bounded (Proposition~\ref{thm:geomdistfinnum}). 
This is the key point of Gromoll and Meyer celebrated Riemannian multiplicity result.

The results are obtained mostly by functional analytical techniques. Using periodic frames along the geodesic
(Section~\ref{sec:spectralflow}), the problem is cast into the language of differential systems in $\R^n$
and studied in the appropriate Sobolev space setting.
Following Bott's ideas, the spectral flow function $\lambda_\gamma$ is then obtained by considering a suitable
complexification of the index form and of the space of infinitesimal variations of the geodesic (Subsections~\ref{sub:basicdata}
and \ref{sub:spflfunction}). The central property of $\lambda_\gamma$, that gives $\spfl(\gamma^{(N)})$ as a sum
of the values of $\lambda_\gamma$ at the $N$-th roots of unity, is proved in Section~\ref{sec:fourier}; using Bott's suggestive terminology,
this is called the \emph{Fourier theorem}. Its proof in the non positive definite case relies heavily on a very special
property of the index form, which is that of being represented by a compact perturbation of a fixed \emph{symmetry} of the Hilbert space of
variations of $\gamma$. By a symmetry of a Hilbert space it is meant a self-adjoint operator $\mathfrak I$ whose
square $\mathfrak I^2$ is the identity. For paths of the form symmetry plus compact, the spectral flow only depends on the
endpoints of the path, which is used in the proof of the Fourier theorem.
The question of continuity of $\lambda_\gamma$, which is quite straightforward in the positive definite case,
is more involved in the general semi-Riemannian case. At points $z\in\mathds S^1\setminus\{1\}$, it is obtained
by showing a perturbation result for the spectral flow of paths of self-adjoint Fredholm operators restricted
to continuous families of closed subspaces of a fixed Hilbert space (Corollary~\ref{thm:flussospettrnondipende}).
The definition and a few basic properties of spectral flow on varying domains are discussed preliminarly in Section~\ref{sec:spectrflowvarying}.
As to the point $z=1$, where Corollary~\ref{thm:flussospettrnondipende} does \emph{not} apply,
we use a certain finite dimensional reduction formula for the spectral flow (Proposition~\ref{thm:BenPic}),
which was proved recently in \cite{BenPic} to show that the spectral flow function has in $z=1$ a sort of \emph{artificial}
discontinuity when $\gamma$ is nondegenerate. The reduction formula is used also in last section, where we obtain the
iteration formula for the spectral flow (Proposition~\ref{thm:spflredfindim}) and we prove estimates on
its growth. For simplicity, in this paper we will only consider orientation preserving closed geodesics; however,
in Subsection \ref{nonorientation} we discuss briefly how to deal with  the non-orientation preserving case.

Future developments of the theory of periodic semi-Riemannian geodesics should include an equivariant
version of the strongly indefinite Morse theory, along the lines of \cite{Wang}. A preliminary important
step would deal with the case of nondegenerate critical orbits; in the context of periodic geodesics,
this would apply to the so-called \emph{bumpy metrics}. Recall that a metric is bumpy if all its closed
geodesics are nondegenerate. Bumpy metrics are generic in the Riemannian setting (see \cite{Abr, Ano, KliTak, Whi});
nothing is known with this respect in the nonpositive definite case.
\end{section}
\begin{section}{Spectral flow on varying domains}
\label{sec:spectrflowvarying}
Let $H$ be a real or complex Hilbert space; we will denote by $\mathrm B(H)$ the Banach algebra of all bounded
operators on $H$, by $\mathrm{GL}(H)$ the open subset of $\mathrm B(H)$ consisting of all isomorphisms,
by $\mathrm O(H)$ the subgroup of $\mathrm{GL}(H)$ consisting of all isometries,
and by $\mathcal F_{\mathrm{sa}}(H)$ the set of all self-adjoint
Fredholm operators on $H$. The adjoint of an operator $T$ on $H$ will be denoted by $T^*$.
Let us recall that the spectral flow is a integer invariant associated to a continuous
path $T:[a,b]\to\mathcal F_{\mathrm{sa}}(H)$, which is:
\begin{itemize}
\item[(i)] fixed endpoint homotopy invariant;
\item[(ii)] additive by concatenation;
\item[(iii)] invariant by \emph{cogredience}, i.e., given two Hilbert spaces $H_1$, $H_2$, a continuous curve $T:[a,b]\to\mathcal F_{\mathrm{sa}}(H_1)$ and a continuous curve
 $M:[a,b]\to\mathrm{Iso}(H_1,H_2)$ of isomorphisms from $H_1$ to $H_2$, then the spectral flow of $T$ on $H_1$
coincides with the spectral flow of $[a,b]\ni t\mapsto M_tT_tM_t^*$ on $H_2$.
\end{itemize}
 We will denote by $\spfl(T,[a,b])$ the spectral flow of the curve $T$; recall that
$\spfl(T,[a,b])$ is a sort of algebraic count of the degeneracy instants of the path $T_t$ as $t$ runs from $a$ to $b$.
Details on the definition and the basic properties of the spectral flow can be found, for instance, in refs.~\cite{BenPic, GiaPicPorCOMPTES, Phillips}.
There are several conventions on how to compute the contribution of the endpoints of the path, in case of
degenerate endpoints; although making a specific choice is irrelevant in the context of the present paper,  
we will follow the convention in \cite{Phillips}.

Property (i) above  holds in fact in a slightly more general form, as follows.
If $h:[a,b]\times[c,d]\to\mathcal F_{\mathrm{sa}}(H)$ is a continuous map such that  $\mathrm{dim}\big(\mathrm{Ker}(h_{a,s})\big)$
and $\mathrm{dim}\big(\mathrm{Ker}(h_{b,s})\big)$ are constant for all $s\in[c,d]$, then the spectral flow of the curve
$[a,b]\ni t\mapsto h_{t,c}$ equals the spectral flow of the curve $[a,b]\ni t\mapsto h_{t,d}$ (see Corollary \ref{thm:flussospettrnondipende}).

We will need to consider paths of Fredholm operators defined on varying domains. Let us consider the following setup:
let $[a,b]\ni t\mapsto H_t$ be a continuous path of closed subspaces of $H$. Recall that this means that, denoting by $P_t:H\to H$ the orthogonal
projection onto $H_t$, then the curve $t\mapsto P_t$ is continuous relatively to the operator norm topology of $\mathrm B(H)$.
For instance, the kernels of a continuous family $t\mapsto F_t$ of surjective bounded linear maps from $H$ to some other Hilbert space
$H'$ form a continuous family of closed subspaces of $H$ (\cite[Lemma~2.9]{asian}).
A simple lifting argument in fiber bundles shows that there exists a continuous curve $t\mapsto\Phi_t\in\mathrm{GL}(H)$ and\footnote{%
In fact, one can find the curve $\Phi_t$ taking values in $\mathrm O(H)$, see \cite{BenPic}.}
a closed subspace $H_\star$ of $H$ such that $\Phi_t(H_\star)=H_t$ for all $t$. We will call the pair $(\Phi,H_\star)$ a
\emph{trivialization} of the path $t\mapsto H_t$.
Assume now that $[a,b]\ni t\mapsto T_t\in\mathrm B(H)$ is a continuous
curve with the property that $P_tT_t\vert_{H_t}:H_t\to H_t$ belongs to $\mathcal F_{\mathrm{sa}}(H_t)$ for all $t$.
Then, given a trivialization $(\Phi,H_\star)$ for $(H_t)_{t\in[a,b]}$, for all $t\in[a,b]$ the
operator $P_\star\Phi_t^*P_tT_t\Phi_t\vert_{H_\star}:H_\star\to H_\star$ belongs
to $\mathcal F_{\mathrm{sa}}(H_\star)$, where $P_\star$ is the orthogonal projection onto $H_\star$.
We can therefore give the following definition:
\begin{defin}
The spectral flow of the path $T$ over the varying domains $(H_t)_{t\in[a,b]}$, denoted by $\spfl\big(T;(H_t)_{t\in[a,b]}\big)$,
is defined as the spectral flow
of the continuous path $[a,b]\ni t\mapsto P_\star\Phi_t^*P_tT_t\Phi_t\vert_{H_\star}$ of self-adjoint Fredholm operators on $H_\star$.
\end{defin}
Invariance by cogredience shows easily that the above definition does not depend on the choice of
the trivialization $(\Phi,H_\star)$ of $(H_t)_{t\in[a,b]}$. Namely, assume that $(\widetilde\Phi,\widetilde H_\star)$
is another trivialization of $(H_t)_{t\in[a,b]}$. Denoting by $P_\star$ (resp.,  $\widetilde P_\star$)
the orthogonal projection onto $H_\star$ (resp., onto $\widetilde H_\star$), and setting
$B_t=P_\star\Phi_t^*P_tT_t\Phi_t\vert_{H_\star}$, $\widetilde B_t=\widetilde P_\star\widetilde\Phi_t^* P_tT_t\widetilde\Phi_t\vert_{\widetilde H_\star}$,
and  $\Psi_t=\Phi_t^{*}\widetilde\Phi_t$,
one has:
\[\widetilde B_t=\big(\Psi_t\vert_{\widetilde H_\star}\big)^*B_t\big(\Psi_t\vert_{\widetilde H_\star}\big)\]
for all $t$, hence $\spfl(B,[a,b])=\spfl(\widetilde B,[a,b])$.
\smallskip

Let us study how the spectral flow varies with respect to the domain.
\begin{lem}\label{thm:lemacont}
Let $[a,b]\ni t\mapsto T_t\in\mathrm B(H)$ be a continuous map and $[c,d]\ni s\mapsto H_s$ be a continuous path of closed
subspaces of $H$, with the property that $P_sT_t\vert_{H_s}\in\mathcal F_{\mathrm{sa}}(H_s)$ for all $s$ and $t$. For all
$s\in[c,d]$, denote by $\mathfrak h_s$ the spectral flow of the path $[a,b]\ni t\mapsto P_sT_t\vert_{H_s}$ of Fredholm operators
on $H_s$. Similarly, for all $t\in[a,b]$ denote by $\mathfrak v_t$ the spectral flow of the (constant) path of
Fredholm operators $T_t$ on the varying domains $(H_s)_{s\in[c,d]}$. Then:
\begin{equation}\label{eq:sommazero}
\mathfrak h_c-\mathfrak h_d=\mathfrak v_a-\mathfrak v_b.
\end{equation}
\end{lem}
\begin{proof}
Choose a trivialization $(\Phi,H_\star)$ for $(H_s)_{s\in[c,d]}$, and define a continuous map $B:[a,b]\times[c,d]\to\mathcal F_{\mathrm{sa}}(H_\star)$ by:
\[B_{s,t}=P_\star\Phi_s^*P_sT_t\Phi_s\vert_{H_\star}.\]
By definition, $\mathfrak v_t=\spfl(s\mapsto B_{s,t},[c,d])$ and $\mathfrak h_s=\spfl(t\mapsto B_{s,t},[a,b])$;
formula \eqref{eq:sommazero} follows immediately from the homotopy invariance and the concatenation additivity of the spectral flow.
\end{proof}
\begin{cor}\label{thm:flussospettrnondipende}
Under the assumptions of Lemma~\ref{thm:lemacont}, if $\Ker\big(P_sT_a\vert_{H_s}\big)$ and $\Ker\big(P_sT_b\vert_{H_s}\big)$
have constant dimension for all $s\in[a,b]$, then the spectral flow of $t\mapsto P_sT_t\vert_{H_s}$ on $H_s$ does not depend on $s$.
\end{cor}
\begin{proof}
This follows easily from the fact that curves of self-adjoint Fredholm operators with kernel of constant dimension
have null spectral flow. Thus, under our assumptions both terms $\mathfrak v_a$ and $\mathfrak v_b$ in \eqref{eq:sommazero} vanish.
\end{proof}
\end{section}
\begin{section}{Spectral flow along a closed geodesic}
\label{sec:spectralflow}
We will recall from \cite{BenPic} the definition of spectral flow along a closed geodesic.

\subsection{Periodic geodesics}
We will consider throughout an $n$-dimensional semi-Riemannian manifold
$(M,g)$, denoting by $\nabla$ the covariant derivative of its Levi--Civita
connection, and by $R$ its curvature tensor, chosen with the sign convention
$R(X,Y)=[\nabla_X,\nabla_Y]-\nabla_{[X,Y]}$.
Let $\gamma:[0,1]\to M$ be a periodic geodesic in $M$, i.e., $\gamma(0)=\gamma(1)$ and
$\dot\gamma(0)=\dot\gamma(1)$. We will assume that $\gamma$ is \emph{orientation
preserving}, which means that the parallel transport along $\gamma$ is orientation preserving;
the non orientation preserving case can be studied similarly, as explained in Subsection~\ref{nonorientation}.
If $M$ is orientable, then every closed geodesic is orientation preserving. Moreover,
given any closed geodesic $\gamma$, its two-fold iteration $\gamma^{(2)}$, defined
by $\gamma^{(2)}(t)=\gamma(2t)$, is always orientation preserving.
We will denote by $\Ddt$ the covariant differentiation of vector fields along $\gamma$;
recall that the \emph{index form} $I_\gamma$ is the bounded symmetric
bilinear form defined on the Hilbert space of all periodic vector fields of Sobolev class $H^1$ along
$\gamma$, given by:
\begin{equation}\label{eq:defindexform}
I_\gamma(V,W)=\int_0^1g\big(\Ddt V,\Ddt W)+g(RV,W)\,\mathrm dt,
\end{equation}
where we set $R=R(\dot\gamma,\cdot)\dot\gamma$.
Closed geodesics in $M$ are the critical points of the geodesic action functional
$f(\gamma)=\frac12\int_0^1g(\dot\gamma,\dot\gamma)\,\mathrm dt$ defined in
the \emph{free loop space} $\Omega M$ of $M$; $\Omega M$ is the Hilbert manifold of
all closed curves in $M$ of Sobolev class $H^1$. The index form $I_\gamma$ is the second
variation of $f$ at the critical point $\gamma$; unless $g$ is positive definite,
the Morse index of $f$ at each non constant critical point is infinite.
The notion of Morse index is replaced by the notion of spectral flow.

Let us denote by $\mathfrak P_\gamma:T_{\gamma(0)}M\oplus T_{\gamma(0)}M\to T_{\gamma(0)}M\oplus T_{\gamma(0)}M$ the
\emph{linearized Poincar\'e map} of $\gamma$, defined by:
\[\mathfrak P_\gamma(v,v')=\big(V(1),\Ddt V(1)\big),\]
where $V$ is the unique Jacobi field along $\gamma$ such that $V(0)=v$ and $\Ddt V(0)=v'$.
Fixed points of $\mathfrak P_\gamma$ correspond to periodic Jacobi fields along $\gamma$.
Moreover, $\mathfrak P_\gamma$ preserves the symplectic form $\varpi$ of $T_{\gamma(0)}M\oplus T_{\gamma(0)}M$
defined by:
\[\varpi\big((v,v'),(w,w')\big)=g(v,w')-g(v',w).\]
\subsection{Periodic frames and trivializations}
Consider a smooth periodic orthonormal frame $\mathbf T$ along $\gamma$, i.e.,
a smooth family $[0,1]\ni t\mapsto T_t$ of isomorphisms:
\begin{equation}\label{eq:periodicframe}
T_t:\R^n\longrightarrow T_{\gamma(t)}M
\end{equation}
 with $T_0=T_1$, and \begin{equation}\label{eq:deltaij}g(T_te_i,T_te_j)=\epsilon_i\delta_{ij},\end{equation}
 where $\{e_i\}_{i=1,\ldots,n}$ is the canonical basis of $\R^n$, $\epsilon_i\in\{-1,1\}$ and
 $\delta_{ij}$ is the Kronecker symbol.
 The existence of such frame is guaranteed by the orientability assumption on the closed
 geodesic.
The pull-back by $T_t$ of the metric $g$ gives a symmetric nondegenerate bilinear form
$G$ on $\R^n$, whose index is the same as the index of $g$; note that this pull-back does not depend
on $t$, by the orthogonality assumption on the frame $\mathbf T$. In the sequel, we will
also denote by $G:\R^n\to\R^n$ the symmetric linear operator defined by $(Gv)\cdot w$;
by \eqref{eq:deltaij}, $G$ satisfies:
\begin{equation}\label{eq:Gsymmetry}
G^2=\mathrm{Id}.
\end{equation}
Moreover, the pull-back of the linearized Poincar\'e map $\mathfrak P_\gamma$ by the isomorphism $T_0\oplus T_0:\R^n\oplus\R^n\to T_{\gamma(0)}M\oplus T_{\gamma(0)}M$
gives a linear endomorphism of $\R^n\oplus\R^n$ that will be denoted by $\mathfrak P$.

 For all $t\in\left]0,1\right]$, define by $\mathcal H^\gamma_t$ the Hilbert space
 of all $H^1$-vector fields $V$ along $\gamma\vert_{[0,t]}$ satisfying:
 \[T_0^{-1}V(0)=T_t^{-1}V(t).\]
 Observe that the definition of $\mathcal H^\gamma_t$ depends on the choice of the periodic frame $\mathbf T$, however,
 $\mathcal H^\gamma_1$, which is the space of all periodic vector fields along $\gamma$, does not depend
 on $\mathbf T$. Although in principle there is no necessity of fixing a specific Hilbert space
 inner product, it will be useful to have one at disposal, and this will be chosen as follows.
 For all $t\in\left]0,1\right]$, consider the Hilbert space:
 \[H^1_{\mathrm{per}}\big([0,t],\R^n\big)=\Big\{\overline V\in H^1\big([0,t],\R^n):\overline V(0)=\overline V(t)\Big\};\]
 a natural Hilbert space inner product in $H^1_{\mathrm{per}}\big([0,t],\R^n\big)$ is given by:
 \begin{equation}\label{eq:innprodhi1per}
 \langle\overline V,\overline W\rangle=\overline V(0)\cdot\overline W(0)+\int_0^t\overline V'(s)\cdot\overline W'(s)\,\mathrm ds,
 \end{equation}
 where $\cdot$ is the Euclidean inner product in $\R^n$.
 The map $\Psi_t:\mathcal H^\gamma_t\to H^1_{\mathrm{per}}\big([0,t],\R^n\big)$ defined by $\Psi_t(V)=\overline V$,
 where $\overline V(s)=T_s^{-1}(V(s))$ is a linear isomorphism; the space $\mathcal H^\gamma_t$ will be endowed with
 the pull-back of the inner product \eqref{eq:innprodhi1per} by the isomorphism $\Psi_t$.
 Denote by $\overline R_t\in\mathrm{End}(\R^n)$ the pull-back by $T_t$ of the endomorphism $R_{\gamma(t)}=R(\dot\gamma,\cdot)\dot\gamma$
 of $T_{\gamma(t)}M$:
 \[\overline R_t=T_t^{-1}\circ R_{\gamma(t)}\circ T_t;\]
 observe that $t\mapsto\overline R_t$ is a smooth map of $G$-symmetric endomorphisms of $\R^n$. Finally, denote by $\Gamma_t\in\mathrm{End}(\R^n)$ the
 \emph{Christoffel symbol} of the frame $\mathbf T$, defined by:
 \[\Gamma_t(v)=T_t^{-1}\big(\Ddt V\big)-\frac{\mathrm d}{\mathrm dt}\overline V(t),\]
 where $\overline V$ is any vector field satisfying $\overline V(t)=v$, and $V=\Psi^{-1}_t(\overline V)$.
A straightforward computation shows that $\Gamma_t$ is $G$-anti-symmetric for all $t$.

 The push-forward  by $\Psi_t$ of the index form $I_\gamma$ on $\mathcal H^\gamma_t$ is given by the bounded
 symmetric bilinear form $\overline I_t$ on $H^1_{\mathrm{per}}\big([0,t],\R^n\big)$ defined by:
 \begin{multline}\label{eq:defoverlineIt}
 \overline I_t(\overline V,\overline W)=\int_0^tG\big(\overline V'(s),\overline W'(s)\big)+G\big(\Gamma_s\overline V(s),\overline W'(s)\big)+
 G\big(\overline V'(s),\Gamma_s\overline W(s)\big)\\+G\big(\Gamma_s\overline V(s),\Gamma_s\overline W(s)\big)+G\big(\overline R_s\overline V(s),\overline W(s)\big)\,\mathrm ds.
 \end{multline}
 Finally, for $t\in\left]0,1\right]$, we will consider the isomorphism \[\Phi_t:H^1_{\mathrm{per}}\big([0,t],\R^n\big)\to H^1_{\mathrm{per}}\big([0,1],\R^n\big),\]
 defined by $\overline V\mapsto\widetilde V$, where $\widetilde V(s)=\overline V(st)$, $s\in[0,1]$.
 The push-forward by $\Phi_t$ of the bilinear form $\overline I_t$ is given by the bounded symmetric
 bilinear form $\widetilde I_t$ on $H^1_{\mathrm{per}}\big([0,1],\R^n\big)$ defined by:
 \begin{multline}\label{eq:deftildeIt}
 \widetilde I_t(\widetilde V,\widetilde W)=\frac1{t^2}\int_0^1G\big(\widetilde V'(r),\widetilde W'(r)\big)+tG\big(\Gamma_{tr}\widetilde V(r),\widetilde W'(r)\big)+
t G\big(\widetilde V'(r), \Gamma_{tr}\widetilde W(r)\big)\\+t^2G\big(\Gamma_{tr}\widetilde V(r),\Gamma_{tr}\widetilde W(r)\big)+t^2G\big(\overline R_{tr}\widetilde V(r),\widetilde W(r)\big)\,\mathrm dr.
  \end{multline}

\subsection{Spectral flow of a periodic geodesic}
For $t\in\left]0,1\right]$, define the Fredholm bilinear form $B_t$ on the Hilbert space $H^1_{\mathrm{per}}\big([0,1],\R^n\big)$
by setting:
\begin{equation}\label{eq:defformaBt}
B_t=t^2\cdot\widetilde I_t.
\end{equation}
From \eqref{eq:deftildeIt}, one sees immediately that the map $\left]0,1\right]\ni t\mapsto B_t$ can be extended continuously to $t=0$ by setting:
\[B_0(\widetilde V,\widetilde W)=\int_0^1G\big(\widetilde V'(r),\widetilde W'(r)\big)\,\mathrm dr.\]
Observe that $\Ker(B_0)$ is $n$-dimensional, and it consists of all constant vector fields.
For all $t\in\left[0,1\right]$, the bilinear form $\widetilde I_t$ on $H^1_{\mathrm{per}}\big([0,1],\R^n\big)$ is
represented with respect to the inner product \eqref{eq:innprodhi1per} by a compact perturbation of the
symmetry $\mathfrak J$ of $H^1_{\mathrm{per}}\big([0,1],\R^n\big)$ given by $\widetilde V\mapsto G\widetilde V$.
\begin{defin}\label{defin:spectral}
The \emph{spectral flow} $\spfl(\gamma)$ of the closed geodesic $\gamma$ is defined as the spectral flow of the continuous path of
Fredholm bilinear forms $[0,1]\ni t\mapsto B_t$ on the Hilbert space $H^1_{\mathrm{per}}\big([0,1],\R^n\big)$.
\end{defin}

It is a non trivial fact that the definition of spectral flow along a closed geodesic
does not depend on the choice of a periodic orthonormal frame along the geodesic.
This result is obtained in \cite{BenPic} by determining an explicit formula
giving the spectral flow in terms of some other integers associated to the geodesic,
such as the \emph{Maslov index} and the concavity index. The Maslov index is a symplectic
invariant, which is computed as an intersection number in the Lagrangian Grasmannian
of a symplectic vector space; it will be denoted by $\iMaslov(\gamma)$.

Recall that a \emph{Jacobi field} along $\gamma$ is a smooth vector field $J$ along $\gamma$ that satisfies the
second order linear equation:
\[\DDdt J(t)=R\big(\dot\gamma(t),J(t)\big)\,\dot\gamma(t),\quad t\in[0,1];\]
let us denote by $\mathcal J_\gamma$ the $2n$-dimensional real vector space of all Jacobi fields along $\gamma$.
Let us introduce the following spaces:
\begin{eqnarray*}
&&\mathcal J_\gamma^{\text{per}}=\Big\{J\in\mathcal J_\gamma:J(0)=J(1),\ \Ddt J(0)=\Ddt J(1)\Big\},\\
&&\mathcal J_\gamma^0=\Big\{J\in\mathcal J_\gamma:J(0)=J(1)=0\Big\},\quad\text{and}\\ &&\mathcal J_\gamma^\star=\Big\{J\in\mathcal J_\gamma:J(0)=J(1)\Big\}.
\end{eqnarray*}
It is well known that $\mathcal J_\gamma^{\text{per}}$ is the kernel of the index form $I_\gamma$ defined in \eqref{eq:defindexform},
while $\mathcal J_\gamma^0$ is the kernel of the restriction of the index form to the space of vector fields
along $\gamma$ vanishing at the endpoints. We denote by $\mathrm n_{\text{per}}(\gamma)$ and $\mathrm n_0(\gamma)$ the
dimensions of $\mathcal J_\gamma^{\text{per}}$ and $\mathcal J_\gamma^0$ respectively. The nonnegative integer
$\mathrm n_{\text{per}}(\gamma)$ is the nullity of $\gamma$ as a periodic geodesic, i.e., the nullity of the Hessian
of the geodesic action functional at $\gamma$ in the space of closed curves. Observe that $\mathrm n_{\text{per}}(\gamma)\ge1$,
as $\mathcal J_\gamma^{\text{per}}$ contains the one-dimensional space spanned by the tangent field $J=\dot\gamma$.
Similarly, $\mathrm n_0(\gamma)$ is the nullity of $\gamma$ as a fixed endpoint geodesic, i.e., it is the nullity of
the Hessian of the geodesic action functional at $\gamma$ in the space of fixed endpoints curves in $M$.
In this case, $\mathrm n_0(\gamma)>0$ if and only if $\gamma(1)$ is conjugate to $\gamma(0)$ along $\gamma$.
The \emph{index of concavity} of $\gamma$, that will be denoted by
$\mathrm i_{\text{conc}}(\gamma)$ is a nonnegative integer invariant associated to
periodic solutions of Hamiltonian systems. In our notations, $\mathrm i_{\text{conc}}(\gamma)$ is equal to
the index of the symmetric bilinear form:
\[(J_1,J_2)\longmapsto g\big(\Ddt J_1(1)-\Ddt J_1(0),J_2(0)\big)\]
defined on the vector space $\mathcal J_\gamma^\star$. It is not hard to show
that this bilinear form is symmetric, in fact, it is given by the restriction of
the index form $I_\gamma$ to $\mathcal J_\gamma^\star$.

\begin{teo}\label{Morseperiodic}
Let $(M,g)$ be a semi-Riemannian manifold and let $\gamma:[0,1]\to M$ be a closed oriented geodesic in $M$.
Then, the spectral flow $\spfl(\gamma)$ is given by the following formula:
\begin{equation}\label{eq:formulaspectralflowperiodic}
\spfl(\gamma)=\Dim\big(\mathcal J_\gamma^{\text{per}}\cap\mathcal J_\gamma^0\big)-\iMaslov(\gamma)-\mathrm i_{\text{conc}}(\gamma)-\mathrm n_-(g),
\end{equation}
where $\mathrm n_-(g)$ is the index of the metric tensor $g$.
\end{teo}
\begin{proof}
See \cite[Theorem~5.6]{BenPic}.
\end{proof}

Formula \eqref{eq:formulaspectralflowperiodic} proves in particular that the definition of spectral flow
for a periodic geodesic $\gamma$ does not depend on the choice of an orthonormal frame along $\gamma$.

\end{section}

\begin{section}{The spectral flow function}
\label{sec:spectralflowfunction}
\subsection{The basic data}\label{sub:basicdata}
As described in Section~\ref{sec:spectralflow}, the choice of a smooth periodic orthonormal frame along a closed orientation preserving semi-Riemannian
geodesic produces the following objects:
\begin{itemize}
\item a non degenerate symmetric bilinear form $G$ on $\R^n$ and a symplectic form $\varpi$ on $\R^n\oplus\R^n$ defined by
$\varpi\big((v,v'),(w,w')\big)=G(v,w')-G(v',w)$;
\item a smooth $1$-periodic curve $\Gamma:\R\to\mathrm{End}(\R^n)$ of $G$-anti-symmetric linear endomorphisms of $\R^n$;
\item a smooth $1$-periodic curve $\overline R:\R\to\mathrm{End}_G(\R^n)$ of $G$-symmetric linear endomorphisms of $\R^n$;
\item a linear endomorphism $\mathfrak P:\R^n\oplus\R^n\to\R^n\oplus\R^n$ that preserves the symplectic form $\varpi$.
\end{itemize}
We will use a complexification of these data. More precisely, $G$ will be extended by sesquilinearity on $\C^n$,
$\Gamma$ and $\overline R$ will be extended to $\C$-linear endomorphisms of $\C^n$, and $\mathfrak P$ will be extended
to $\C$-linear endomorphisms of $\C^{2n}$.
Let $\mathcal H$ be the complex Hilbert space $H^1\big([0,1],\C^n)$ endowed with inner product:
\[\langle \widetilde V,\widetilde W\rangle=\widetilde V(0)\cdot \widetilde W(0)+\int_0^1\widetilde V'(r)\cdot\widetilde W'(r)\,\mathrm dr,\]
where now $\cdot$ denotes the canonical Hermitian product in $\C^n$: $v\cdot w=\sum_{j=1}^nv_j\overline w_j$.
Given a unit complex number $z$, let $\mathcal H_z$ denote the Hilbert subspace of $\mathcal H$ defined by:
\begin{equation}\label{eq:defHz}
\mathcal H_z=\big\{\widetilde V\in\mathcal H:\widetilde V(1)=z\widetilde V(0)\big\};
\end{equation}
moreover, we will denote by $\mathcal H_o$ the subspace:
\begin{equation}\label{eq:defHo}
\mathcal H_o=\big\{\widetilde V\in\mathcal H:\widetilde V(0)=\widetilde V(1)=0\big\}.
\end{equation}
For $t\in[0,1]$, let $B_t:\mathcal H\times\mathcal H\to\C$ denote the bounded Hermitian form defined by:
\begin{multline}\label{eq:defBtRnCn}
B_t(\widetilde V,\widetilde W)=\int_0^1G\big(\widetilde V'(r),\widetilde W'(r)\big)+tG\big(\Gamma_{tr}\widetilde V(r),\widetilde W'(r)\big)+
t G\big(\widetilde V'(r),\Gamma_{tr}\widetilde W(r)\big)\\+t^2G\big(\Gamma_{tr}\widetilde V(r),\Gamma_{tr}\widetilde W(r)\big)+
t^2G\big(\overline R_{tr}\widetilde V(r),\widetilde W(r)\big)\,\mathrm dr.
\end{multline}
With the above data, the Jacobi equation along $\gamma$ gives the following second order linear homogeneous equation for vector fields
in $\C^n$:
\begin{equation}\label{eq:JacobieqCn}
\widetilde V''(r)+2\Gamma_r\widetilde V'(r)+(\Gamma'_r+\Gamma_r^2-\overline R_r)\widetilde V(r)=0.
\end{equation}
The linear map $\mathfrak P:\C^n\oplus\C^n\to\C^n\oplus\C^n$ is given by:
\[\mathfrak P(v,v')=\big(\widetilde V(1),\widetilde V'(1)+\Gamma_0\widetilde V(1)\big),\]
where $\widetilde V\in C^2\big([0,1],\C^n\big)$ is the unique solution of \eqref{eq:JacobieqCn} satisfying $\widetilde V(0)=v$ and
$\widetilde V'(0)=v'-\Gamma_0v$. We observe that $t\to B_t$ can extended in the obviuos way for $t\in[0,+\infty)$, and $B_N$ corresponds with the index form associated to the $N$-th iterate $\gamma^{(N)}$.

\subsection{The spectral flow function on the circle}\label{sub:spflfunction}
As in Section~\ref{sec:spectralflow}, it is easy to see that for all $t\in[0,1]$ and all $z\in\mathds S^1$,
the restriction of the Hermitian form $B_t$ to $\mathcal H_z\times\mathcal H_z$ is represented by a compact
perturbation of the symmetry $\mathfrak J$ of $\mathcal H$ defined by $\widetilde V\mapsto G\widetilde V$;
observe that each $\mathcal H_z$ is invariant by $\mathfrak J$.
\begin{defin}\label{thm:defspflofunction}
The \emph{spectral flow function} $\lambda_\gamma:\mathds S^1\to\N$ is defined by:
\[\lambda_\gamma(z)=\text{spectral flow of $[0,1]\ni t\mapsto B_t$ on the space $\mathcal H_z\times\mathcal H_z$}.\]
\end{defin}
Let us recall that the spectral flow of a continuous path of symmetric bilinear forms $B$ on a real Hilbert space $H$ equals the spectral
flow of the continuous path of Fredholm Hermitian forms obtained by taking the sesquilinear extension of $B$ on the complexified Hilbert space
$H^\C$.
In particular, $\lambda_\gamma(1)=\mathfrak{sf}(\gamma)$.

More generally, given an integer $N\ge1$, we can define $\lambda_\gamma(z;N)$ as the spectral flow of the
path $t\mapsto B_{Nt}$ on $\mathcal H_z\times\mathcal H_z$. Then:
\begin{equation}\label{eq:lambda1Nsf}
\lambda_\gamma(1,N)=\mathfrak{sf}\big(\gamma^{(N)}\big),
\end{equation}
where $\gamma^{(N)}$ is the $N$-th iterate of $\gamma$.

It will also be useful to introduce the following notation:
\begin{equation}\label{eq:deflambda0}
\lambda_\gamma^o=\text{spectral flow of $[0,1]\ni t\mapsto B_t$ on the space $\mathcal H_o\times\mathcal H_o$};
\end{equation}
it is shown in \cite{GiaPicPor} that
\begin{equation}\label{spectralflow0}
\lambda_\gamma^o=\mathrm n_0(\gamma)-\mathrm n_-(g)-\iMaslov(\gamma).
\end{equation}
\subsection{Continuity and jumps of the spectral flow function}
The reader will easily check by an immediate partial integration argument that equation \eqref{eq:JacobieqCn} characterizes
the elements in the kernel of the Hermitian form $B_1$ in \eqref{eq:defBtRnCn}; more precisely:
\begin{lem}\label{thm:nucleoB1Poincare}
$\widetilde V\in\mathcal H$ is in the kernel of $B_1:\mathcal H_z\times\mathcal H_z\to\C$ if and only if it satisfies
\eqref{eq:JacobieqCn} and the boundary conditions:
\[\widetilde V(1)=z\cdot\widetilde V(0),\quad\text{and}\quad \widetilde V'(1)=z\cdot\widetilde V'(0).\]
Thus, $B_1$ is degenerate on $\mathcal H_z$ if and only if $z$ is in the spectrum of the linearized Poincar\'e map
$\mathfrak P_\gamma$, and $\mathrm{dim}\big(\mathrm{Ker}(B_1\vert_{\mathcal H_z\times\mathcal H_z})\big)=\mathrm{dim}\big(\mathrm{Ker}(\mathfrak P_\gamma-z\cdot\mathrm{Id})\big)$.\qed
\end{lem}

\begin{lem}
The map $\mathds S^1\ni z\mapsto\mathcal H_z\subset\mathcal H$ is a continuous\footnote{In fact, the very same argument shows
that $z\mapsto\mathcal H_z$ is a real analytic map. The same conclusion holds in Lemma~\ref{thm:J2gammacont}.} map of closed subspaces of $\mathcal H$.
\end{lem}
\begin{proof}
$\mathcal H_z$ is the kernel of a continuous family $F_z:\mathcal H\to\C^n$ of \emph{surjective} bounded linear maps,
defined by $F_z(\widetilde V)=\widetilde V(1)-z\widetilde V(0)$ (see \cite[Lemma~2.9]{asian}).
\end{proof}
\begin{prop}\label{thm:lambdagammaconst}
The spectral flow function $\lambda_\gamma$ is constant on every connected subset $\mathcal A$ of $\mathds S^1\setminus\{1\}$ that does not contain
elements in the spectrum of the linearized Poincar\'e map $\mathfrak P_\gamma$.
\end{prop}
\begin{proof}
This follows easily from Corollary~\ref{thm:flussospettrnondipende}. By Lemma~\ref{thm:nucleoB1Poincare}, the assumption on the spectrum
of the Poincar\'e map says that $B_1$ does not degenerates on $\mathcal H_z$, for all $z\in\mathcal A$.
Moreover, it is easy to see that $B_0$ is nondegenerate on $\mathcal H_z$ for all $z\ne1$ (while the kernel of $B_0\vert_{\mathcal H_1\times\mathcal H_1}$
consists of all constant maps, and it has dimension $n$). This concludes the proof.
\end{proof}
We conclude that $\lambda_\gamma$ has a finite number of jumps on $\mathds S^1$, that can occur only at those points in the spectrum
of $\mathfrak P_\gamma$ that lie in $\mathds S^1$ or at $z=1$.
\smallskip

Let us study now the behavior of $\lambda_\gamma$ around $z=1$; the result of Proposition~\ref{thm:lambdagammaconst}
cannot be extended to $z=1$, because $B_0$ is always degenerate on $\mathcal H_1$.
We will show that, unless the metric tensor $g$ is
positive definite, then $\lambda_\gamma$ is indeed discontinuous at $z=1$. To this aim, we will need to determine an alternative
description of the function $\lambda_\gamma$ based on a finite dimensional reduction for the computation
of the spectral flow.

\subsection{A finite dimensional reduction}
\label{sub:findimreduction}
By a \emph{symmetry} of a Hilbert space $H$ we mean a bounded self-adjoint operator $\mathfrak J$ on $H$ satisfying $\mathfrak J^2=1$.
For paths $[a,b]\ni t\mapsto T_t\in\mathcal F_{\mathrm{sa}}(H)$ of Fredholm self-adjoint operators of the form $T_t=\mathfrak J+K_t$,
where $\mathfrak J$ is a fixed symmetry of $H$ and $K_t$ is a compact self-adjoint operator on $H$ for all $t$, the spectral flow
$\spfl(T,[a,b])$ depends only on the endpoints $T_a$ and $T_b$. More precisely, $\spfl(T,[a,b])$ is the relative dimension
of the generalized negative spaces of $T_t$ at $t=a$ and at $t=b$. We want to compare the spectral flow of a path $T$ on $H$
with the spectral flow of its restriction to a finite codimensional subspace of $H$.
Let us recall the following result from \cite{BenPic}:
\begin{prop}\label{thm:BenPic}
Let $T:[a,b]\to\mathcal F_{\mathrm{sa}}(H)$ be a continuous curve where each $T_t$ is a compact perturbation of a fixed symmetry
$\mathfrak J$ of $H$, and let $\mathcal V\subset H$ be a closed subspace of finite codimension in $H$. Set
$B_t=\langle T_t\cdot,\cdot\rangle$ and $\mathcal V_t=(T_t\mathcal V)^\perp=\mathcal V^{\perp_{B_t}}$; then:
\begin{multline}\label{eq:formteo4.5}
\spfl(T,[a,b])-\spfl\big(P_{\mathcal V}T\vert_{\mathcal V},[a,b]\big)=\mathrm n_-\big(B_a\vert_{\mathcal V_a\times\mathcal V_a}\big)+\mathrm{dim}(\mathcal V\cap\mathcal V_a)-
\mathrm{dim}\big(\mathcal V\cap\mathrm{Ker}(B_a)\big)\\-\mathrm n_-\big(B_b\vert_{\mathcal V_b\times\mathcal V_b}\big)-\mathrm{dim}(\mathcal V\cap\mathcal V_b)+
\mathrm{dim}\big(\mathcal V\cap\mathrm{Ker}(B_b)\big),
\end{multline}
where $\mathrm n_-$ denotes the index of a Hermitian (or symmetric in the real case) bilinear form.
\end{prop}
\begin{proof}
See \cite[Theorem~4.3]{BenPic}.
\end{proof}
We apply this result to the path of bilinear forms $[0,1]\ni t\mapsto B_t$ given in \eqref{eq:defBtRnCn}, to the Hilbert space $H=\mathcal H_z$ in \eqref{eq:defHz}
and to the closed finite codimensional subspace $\mathcal V=\mathcal H_o$ defined in \eqref{eq:defHo}, obtaining:
\begin{prop}\label{thm:formularidotta}
For all $z\in\mathds S^1$, the following equality holds:\footnote{Here $\delta_{z,1}$ is the Kronecker symbol, equal to $1$ if $z=1$ and to $0$ otherwise.}
\begin{equation}\label{eq:rellambdagammalambdao}
\lambda_\gamma(z)-\lambda_\gamma^o=(1-\delta_{z,1})\cdot\mathrm n_-(g)-\mathrm n_0(\gamma)+\mathrm{dim}\big(\mathbb J_\gamma^{(1)}(z)\big)-
\mathrm n_-(b_z),
\end{equation}
where $\mathbb J_\gamma^{(1)}(z)$ is the finite dimensional vector space:
\begin{multline*}
\mathbb J_\gamma^{(1)}(z)=\mathcal H_o\cap\mathrm{Ker}(B_1\vert_{\mathcal H_z\times\mathcal H_z})\\=
\big\{\widetilde V\in C^2\big([0,1],\C^n\big):\widetilde V\ \text{solution of \eqref{eq:JacobieqCn}},\ \widetilde V(0)=\widetilde V(1)=0,\  \widetilde V'(1)=z\widetilde V'(0)\big\}
\end{multline*}
and $b_z$ is the Hermitian form on the finite dimensional
space:
\[\mathbb J^{(2)}_\gamma(z)=\big\{\widetilde V\in C^2\big([0,1],\C^n\big):\widetilde V\ \text{solution of \eqref{eq:JacobieqCn}},\ \widetilde V(1)=z\widetilde V(0)\big\},\]
given by the restriction of $B_1$, or, more explicitely:
\[b_z(\widetilde V,\widetilde W)=G\big(\bar z\widetilde V'(1)-\widetilde V'(0),\widetilde W(0)\big).\]
\end{prop}
\begin{proof}
Formula~\eqref{eq:rellambdagammalambdao} follows directly by applying Proposition~\ref{thm:BenPic} to the above setup.
It is easily obtained after checking the following identities:
\begin{itemize}
\item $\mathcal H_o^{\perp_{B_0}}\cap\mathcal H_z=\big\{\widetilde V:\widetilde V(t)=A(z-1)t+A,\ \text{for some $A\in\C^n$}\big\}$,
thus the index of the restriction of $B_0$ to such space is equal to $0$ if $z=1$ and is equal to $\mathrm n_-(G)=\mathrm n_-(g)$ if $z\ne1$;
\smallskip

\item $\mathcal H_o\cap\big(\mathcal H_o^{\perp_{B_0}}\cap\mathcal H_z\big)=\{0\}$;
\smallskip

\item $\Ker\big(B_0\vert_{\mathcal H_z\times\mathcal H_z}\big)=\big\{\widetilde V:\widetilde V''=0,\ \big(\widetilde V(1),\widetilde V'(1)\big)=z\cdot
\big(\widetilde V(0),\widetilde V'(0)\big)\big\}$, such space is $\{0\}$ if $z\ne1$, and it consists of constant functions if $z=1$;
\smallskip

\item $\mathcal H_o\cap\big(\Ker\big(B_0\vert_{\mathcal H_z\times\mathcal H_z}\big)\big)=\{0\}$;
\smallskip

\item $\Ker\big(B_1\vert_{\mathcal H_z\times\mathcal H_z}\big)=\big\{\widetilde V\ \text{solution of \eqref{eq:JacobieqCn}}:\big(\widetilde V(1),\widetilde V'(1)\big)=z\cdot
\big(\widetilde V(0),\widetilde V'(0)\big)\big\}$;
\smallskip

\item $\mathcal H_o\cap\Ker\big(B_1\vert_{\mathcal H_z\times\mathcal H_z}\big)=\mathbb J_\gamma^{(1)}(z)$;
\smallskip

\item $\mathcal H_o^{\perp_{B_1}}\cap\mathcal H_z=\mathbb J_\gamma^{(2)}(z)$;
\smallskip

\item $\mathcal H_o\cap\mathcal H_o^{\perp_{B_1}}=\mathrm{Ker}\big(B_1\vert_{\mathcal H_o\times\mathcal H_o})=\big\{\widetilde V\ \ \text{solution of \eqref{eq:JacobieqCn}},\ \widetilde V(0)=\widetilde V(1)=0\big\}$.\qedhere
\end{itemize}
\end{proof}

We recall that the term $\mathrm n_0(\gamma)$ in \eqref{eq:rellambdagammalambdao} is the nullity of $\gamma$ as a \emph{fixed endpoints geodesic}, i.e., the multiplicity of $\gamma(1)$ as
conjugate point to $\gamma(0)$ along $\gamma$. In particular, it coincides with the dimension of the space:
\begin{equation}\label{eq:jacobizeriestremi}
\big\{\widetilde V\in C^2\big([0,1],\C^n\big):\widetilde V\ \text{solution of \eqref{eq:JacobieqCn}},\ \widetilde V(0)=\widetilde V(1)=0\big\}.
\end{equation}
 Observe that for all $z\in\mathds S^1$, the space $\mathbb J^{(1)}_\gamma(z)$ is contained in
\eqref{eq:jacobizeriestremi}; it follows that
\begin{equation}\label{eq:terminenegativo}
\phantom{\qquad\forall\,z\in\mathds S^1.}-\mathrm n_0(\gamma)+\mathrm{dim}\big(\mathbb J_\gamma^{(1)}(z)\big)\le0,\qquad\forall\,z\in\mathds S^1.
\end{equation}

\begin{lem}\label{thm:J2gammacont}
The map $z\mapsto\mathbb J^{(2)}_\gamma(z)\subset\mathcal H$ is continuous at those points $z\in\mathds S^1$ that are not in the spectrum of $\mathfrak P$.
\end{lem}
\begin{proof}
It suffices to show that $z\mapsto\mathbb J^{(2)}_\gamma(z)$ is a continuous family of closed subspaces of the finite dimensional
closed subspace $S=\big\{\widetilde V\in C^2\big([0,1],\C^n\big):\widetilde V\ \text{is solution of \eqref{eq:JacobieqCn}}\big\}$ of $\mathcal H$.
If we identify $S\cong\C^n\oplus\C^n$ via the map $\widetilde V\mapsto\big(\widetilde V(0),\widetilde V'(0)\big)$, then
$\mathbb J^{(2)}_\gamma(z)$ is the kernel of the linear map $L_z=\pi_1\circ(\mathfrak P-z\cdot\mathrm I):\C^n\oplus\C^n\to\C^n$, where $\pi_1:\C^n\oplus\C^n\to\C^n$
is the projection on the first summand. Clearly, $z\mapsto L_z$ is continuous.
If $z\in\mathds S^1$ is not in the spectrum of $\mathfrak P$, then $\mathfrak P-z\cdot\mathrm I$ is an isomorphism, and thus $L_z$ is surjective, which concludes
the proof (see \cite[Lemma~2.9]{asian}).
\end{proof}
\begin{cor}\label{thm:saltoinz=1}
Let $\mathcal A$ be a connected subset of $\mathds S^1$ that contains $z=1$, and that does not contain any eigenvalue of the linearized
Poincar\'e map $\mathfrak P_\gamma$. Then,
$\lambda_\gamma$ is constant equal to $\lambda_\gamma(1)+\mathrm n_-(g)$ on $\mathcal A\setminus\{1\}$.
\end{cor}
\begin{proof}
Using formula~\eqref{eq:rellambdagammalambdao}, the reader will easily convince himself that the statement
is equivalent to proving that the quantity $\mathrm{dim}\big(\mathbb J_\gamma^{(1)}(z)\big)-
\mathrm n_-(b_z)$ is constant on every connected subset $\mathcal A$ of $\mathds S^1$ that does not contain
elements in the spectrum of $\mathfrak P_\gamma$. This follows immediately from the continuity of the subspaces $\mathcal A\ni z\mapsto \mathbb J^{(2)}_\gamma(z)$
proved in Lemma~\ref{thm:J2gammacont}, and from the following observations:
\begin{itemize}
\item[(a)]  $\Ker(b_z)=\big\{\widetilde W\ \text{solution of \eqref{eq:JacobieqCn}}:\widetilde W(0)=\widetilde W(1)=0\big\}$, thus $\Ker(b_z)$
does not depend on $z\in\mathcal A$, and $\mathrm n_-(b_z)$ is constant on $\mathcal A$;
\smallskip

\item[(b)] $\mathbb J_\gamma^{(1)}(z)\subset \Ker(B_1\vert_{\mathcal H_z\times\mathcal H_z})$, hence $\mathrm{dim}\big(\mathbb J_\gamma^{(1)}(z)\big)=0$
for all $z\in\mathcal A$.
\end{itemize}
In order to prove (a), note that if $z$ is not in the spectrum of $\mathfrak P_\gamma$, then the map $\mathbb J_\gamma^{(2)}(z)\ni\widetilde V\mapsto\bar z\widetilde V'(1)-\widetilde V'(0)\in\C^n$
is an isomorphism. It follows immediately from the definition of $b_z$ that $\Ker(b_z)=\big\{\widetilde W\in\mathbb J_\gamma^{(2)}(z):\widetilde W(0)=0\big\}$,
which gives the desired conclusion.
\end{proof}
\begin{rem}
Note that, when $1$ is not in the spectrum of $\mathfrak P_\gamma$,  the discontinuity at $z=1$ of $\lambda_\gamma$
occurs only in the non Riemannian case.
\end{rem}
\subsection{Non orientation preserving closed geodesics}\label{nonorientation}
When the geodesic $\gamma$ is non orientation preserving, the definition of the spectral flow given in
Definition~\ref{defin:spectral} does not apply because there is no periodic orthonormal frame along the
geodesic. Let us now indicate briefly how to modify the construction in order to get a well defined notion
of spectral flow also in this case. One can consider an arbitrary smooth orthonormal frame $\mathbf T=(T_t)_{t\in[0,1]}$
along the geodesic as in \eqref{eq:periodicframe}, which will \emph{not} satisfy $T_0=T_1$; denote by $S\in\mathrm{GL}(\C^n)$ the complexification of the isomorphism $T_1^{-1}T_0$.
Then, it would be natural to define the spectral flow $\spfl(\gamma)$ as the spectral flow of the path of Fredholm bilinear forms
$[0,1]\ni t\mapsto B_t$ given in \eqref{eq:defformaBt} on the space:
\[H^1_S\big([0,1],\C^n\big)=\Big\{\widetilde V\in H^1\big([0,1],\C^n\big):\widetilde V(1)=S\widetilde V(0)\Big\}.\]
(compare with Definition~\ref{defin:spectral}). However, with this definition, using Proposition~\ref{thm:BenPic} one checks easily
that formula \eqref{eq:formulaspectralflowperiodic} will not hold in general. More precisely, the right hand side of
 \eqref{eq:formulaspectralflowperiodic} will contain an extra term, given by the index of the restriction of the metric
 tensor $g$ on the image of the operator $S-\mathrm{Id}$.
 This is proved easily with the help of Proposition~\ref{thm:BenPic}, as in the proof of
 Proposition~\ref{thm:formularidotta}.
 Namely, in this case, the $B_0$-orthogonal space
 to $H_0^1\big([0,1],\C^n\big)$ in $H_S^1\big([0,1],\C^n\big)$ is given by the $n$-dimensional space of all affine maps $\widetilde V(t)=(S-\mathrm{Id})Bt+B$, with $B\in\C^n$. The restriction of $B_0$ to such space has index equal to the index of the restriction of
 $g$ to the image of $S-\mathrm{Id}$.
 Note that $S=\mathrm{Id}$ in the orientation preserving case.
 Thus, one way to make the definition independent on the orthogonal frame would be to restrict to frames
 $\mathbf T$ for which the operator $S=T_1^{-1}T_0$ is such that $g$ is positive semi-definite on the image
 of $S-\mathrm{Id}$. This is always possible when $g$ is not negative definite, by a simple linear algebra argument.
 Or, more simply, one could define $\spfl(\gamma)$ as the difference between the spectral flow of the path
 $t\mapsto B_t$ on $H^1_S\big([0,1],\C^n\big)$ minus the index of the restriction of $g$ to $\mathrm{Im}(S-\mathrm{Id})$, which by
 what has been observed, is a quantity independent of the frame. With such definition, formula \eqref{eq:formulaspectralflowperiodic}
 holds also in the non orientation preserving case, and the entire theory developed in this paper carries over
 to the non orientable case.
\subsection{On the jumps of the spectral flow function}
By the results of Proposition~\ref{thm:lambdagammaconst} and Corollary~\ref{thm:saltoinz=1}, we know that
$\lambda_\gamma$ is a piecewise constant function on $\mathds S^1$, whose jumps occur at the eigenvalues of
the linearized Poincar\'e map and at $1$. Thus, $\lambda_\gamma$ is uniquely determined once we know its value
at some given point, say at $z=1$, and the value of the jump of $\lambda_\gamma$ at each discontinuity point.
When $1$ is not an eigenvalue of $\mathfrak P_\gamma$, then the jump of $\lambda_\gamma$ at
$1$ is $\mathrm n_-(g)$, as proved in Corollary~\ref{thm:saltoinz=1}, but $\lim\limits_{\theta\to0^-}\lambda_\gamma(e^{i\theta})=\lim\limits_{\theta\to0^+}\lambda_\gamma(e^{i\theta})$.
In order to give an upper estimate for the jumps at the eigenvalues of $\mathfrak P_\gamma$, we need
the following:
\begin{lem}\label{thm:stimaspflnucleo}
Let $I\ni t\mapsto B_t$ be a continuous curve of Hermitian  forms on a Hilbert space $H$, and
let $I\ni t\mapsto H_t$ be a continuous family of closed subspaces of $H$, such that $B_t\vert_{H_t\times H_t}$
is Fredholm for all $t$.
Assume that a point $t_0$ in the interior of $I$ is an isolated degeneracy instant for $B_t\vert_{H_t\times H_t}$. Then, for
$\varepsilon>0$ small enough, $\big\vert\spfl\big(B,(H_t)_{t\in[t_0-\varepsilon,t_0+\varepsilon]}\big)\big\vert\le\dim\big[\Ker\big(B_{t_0}\vert_{H_{t_0}\times H_{t_0}}\big)\big]$.
\end{lem}
\begin{proof}
Using a trivialization for the path $I\ni t\mapsto H_t$, it suffices to prove the result for the spectral flow
of a continuous path of Fredholm Hermitian forms $t\mapsto B_t$ on a fixed Hilbert space $H$.
When $H$ is finite dimensional, in which case the spectral flow is the variation of the index function, the
result is elementary and well known. The infinite dimensional case can be reduced to the finite dimensional one by
an argument of functional calculus. More precisely, for all $t$, let $T_t$ be the self-adjoint Fredholm operator
such that $B_t=\langle T_t\cdot,\cdot\rangle$; let $\varepsilon>0$ be small enough so that
the spectrum of $T_{t_0}$ has empty intersection with $\left]-2\varepsilon,2\varepsilon\right[\setminus\{0\}$.
Then, for $t$ sufficiently close to $t_0$, the spectrum of $T_t$ has empty intersection with
$\left[-\varepsilon,\varepsilon\right]\setminus\{0\}$. Thus, denoting by $\chi_{[-\varepsilon,\varepsilon]}$ the
characteristic function of $[-\varepsilon,\varepsilon]$, the map $t\mapsto P_t=\chi_{[-\varepsilon,\varepsilon]}(T_t)$
is a map of finite rank projections which is continuous near $t_0$; the image $H_t$ of $P_t$ is $T_t$-invariant, so that
$\Ker(B_t)=\Ker(B_t\vert_{H_t\times H_t})$.
By definition (see \cite{Phillips}), the spectral flow through $t_0$ of the path $T$ is the variation of index
of the restriction of $B_t$ to $H_t\times H_t$, which reduces the general case to a finite dimensional one.
\end{proof}
With this, we are now able to prove that the jump of $\lambda_\gamma$ at $z\in\mathds S^1$ can be estimated with the dimension
of the $z$-eigenspace of $\mathfrak P_\gamma$.
\begin{cor}\label{thm:stimasalti}
Let $e^{i\theta_0}$ be an eigenvalue of $\mathfrak P_\gamma$. Then:
\[\Big\vert\lim_{\theta\to0^+}\lambda_\gamma\big(e^{i(\theta_0+\theta)}\big)-\lim_{\theta\to0^-}\lambda_\gamma\big(e^{i(\theta_0+\theta)}\big)\Big\vert\le
\mathrm{dim}\big(\Ker(\mathfrak P_\gamma-e^{i\theta_0}\cdot\mathrm I)\big).\]
\end{cor}
\begin{proof}
By the concatenation additivity of the spectral flow and as $B_0|_{\mathcal H_z\times\mathcal H_z}$ is non-degenerate for $z\not=1$, for $\theta>0$ small enough the difference
$\lambda_\gamma\big(e^{i(\theta_0+\theta)}\big)-\lambda_\gamma\big(e^{i(\theta_0+\theta)}\big)$ is equal to the
spectral flow of the constant Fredholm Hermitian form $B_1$  on the continuous curve of closed subspaces
$[-\theta,\theta]\ni\rho\mapsto\mathcal H_{e^{i(\theta_0+\rho)}}$. By assumption, there is an isolated degeneracy instant
of $B_1$ at $\rho=0$. By Lemma~\ref{thm:stimaspflnucleo}, the jump of $\lambda_\gamma$ at $e^{i\theta_0}$ is less than
or equal to the dimension of $\Ker\big(B_1\vert_{\mathcal H_{e^{i\theta_0}}\times\mathcal H_{e^{i\theta_0}}}\big)$;
this is equal to $\mathrm{dim}\big(\Ker(\mathfrak P_\gamma-e^{i\theta_0}\cdot\mathrm I)\big)$   by Lemma~\ref{thm:nucleoB1Poincare}.
\end{proof}

\end{section}

\begin{section}{The Fourier Theorem}\label{sec:fourier}
We will now fix an integer $N\ge1$ and we set:
\[\omega=e^{\frac{2\pi i}N}.\]
For all $k\ge0$, given $\widetilde V_k\in\mathcal H_{\omega^k}$ we will assume that $\widetilde V_k$ is extended to
a continuous map $\widetilde V_k:\R\to\C^n$ by setting:
\[\phantom{\quad m\in\Z,\ t\in\left[0,1\right[.}\widetilde V_k(t+m)=\omega^{km}\widetilde V_k(t),\quad m\in\Z,\ t\in\left[0,1\right[.\]
\begin{lem}\label{thm:lem1}
The map $\Phi_N:\mathcal H_1\to\bigoplus\limits_{k=0}^{N-1}\mathcal H_{\omega^k}$ defined by:
\[\Phi_N(\widetilde V)=(\widetilde V_0,\ldots,\widetilde V_{N-1}),\]
where:
\begin{equation}\label{eq:defVk}
\widetilde V_k(t)=\frac1N\sum_{j=0}^{N-1}\omega^{-jk}\widetilde V\left(\frac{t+j}N\right),\quad t\in[0,1],
\end{equation}
is a linear isomorphism, whose inverse $\Psi_N:\bigoplus\limits_{k=0}^{N-1}\mathcal H_{\omega^k}\to\mathcal H_1$
is given by:
\[\Psi_N(\widetilde V_0,\ldots,\widetilde V_{N-1})=\widetilde V,\]
where
\begin{equation}\label{eq:defV}
\widetilde V(t)=\sum_{k=0}^{N-1}\widetilde V_k(tN).
\end{equation}
\end{lem}
\begin{proof}
A matter of straightforward calculations, based on the identity:
\[\sum_{j=0}^{N-1}\omega^{jk}=\begin{cases}N,&\text{if $k\equiv0\mod N$;}\\ 0,&\text{otherwise.}\end{cases}\]
First, one needs to prove that $\Phi_N$ is well defined, i.e., that the map $V_k$ in \eqref{eq:defVk}
belongs to $\mathcal H_{\omega^k}$:
\begin{multline*}
\widetilde V_k(1)=\frac1N\sum_{j=0}^{N-1}\omega^{-kj}\widetilde V\left(\tfrac{1+j}N\right)=
\frac1N\sum_{j=1}^{N}\omega^{-k(j-1)}\widetilde V\left(\tfrac{j}N\right)\\=
\omega^k\left[\frac1N\sum_{j=1}^N\omega^{-kj}\widetilde V\left(\tfrac jN\right)\right]=
\frac{\omega^k}N\left[\widetilde V(1)+\sum_{j=1}^{N-1}\omega^{-kj}\widetilde V\left(\tfrac jN\right)\right]\\
=\frac{\omega^k}N\left[\sum_{j=0}^{N-1}\omega^{-kj}\widetilde V\left(\tfrac jN\right)\right]=\omega^k\widetilde V_k(0),
\end{multline*}
i.e., $\widetilde V_k\in\mathcal H_{\omega^k}$. Similarly, $\Psi_N$ is well defined. Clearly, $\Phi_N$ and $\Psi_N$ are linear and bounded.

In order to check that $\Psi_N$ is a right inverse for $\Phi_N$, set $(\widetilde W_0,\ldots,\widetilde W_{N-1})=\Phi_N\circ\Psi_N(\widetilde V_0,\ldots,\widetilde V_{N-1})$.
Then, for all $k=0,\ldots,N-1$ and all $t\in[0,1]$:
\begin{multline*}
\widetilde W_k(t)=
\frac1N\sum_{j=0}^{N-1}\sum_{l=0}^{N-1}\omega^{-kj}\widetilde V_l(t+j)=
\frac1N\sum_{j=0}^{N-1}\sum_{l=0}^{N-1}\omega^{-kj}\omega^{lj}\widetilde V_l(t)\\=
\frac1N\sum_{l=0}^{N-1}\left(\sum_{j=0}^{N-1}\omega^{j(l-k)}\right)\widetilde V_l(t)=\sum_{l=0}^{N-1}\delta_{l,k}\widetilde V_l(t)=
\widetilde V_k(t),
\end{multline*}
i.e., $\Phi_N\circ\Psi_N$ is the identity of $\bigoplus\limits_{k=0}^{N-1}\mathcal H_{\omega^k}$.

Finally, to check that $\Psi_N$ is a left inverse for $\Phi_N$, set $\widetilde V=\Psi_N\circ\Phi_N(\widetilde W)$ and
compute:
\begin{multline*}
\widetilde V(t)=\frac1N\sum_{k=0}^{N-1}\sum_{j=0}^{N-1}\omega^{-kj}\widetilde W(t+j)\\=
\frac1N\sum_{k=0}^{N-1}\sum_{j=0}^{N-1}\omega^{-kj}\widetilde W(t)=\sum_{k=0}^{N-1}\delta_{k,0}\widetilde W(t)=\widetilde W(t),
\end{multline*}
for all $t\in[0,1]$. Thus, $\Psi_N\circ\Phi_N$ is the identity of $\mathcal H_1$, which concludes the proof.
\end{proof}
\begin{lem}\label{thm:lem2}
Given $\widetilde V,\widetilde W\in\mathcal H_1$, set: \[\Phi_N(\widetilde V)=(\widetilde V_0,\ldots,\widetilde V_{N-1}),\quad
\text{and}\quad \Phi_N(\widetilde W)=(\widetilde W_0,\ldots,\widetilde W_{N-1}),\] with $\widetilde V_k,\widetilde W_k\in\mathcal H_{\omega^k}$
for all $k=0,\ldots,N-1$. Then,  the following identities hold:
\[B_0(\widetilde V,\widetilde W)=N^2\sum_{k=0}^{N-1}B_0(\widetilde V_k,\widetilde W_k),\quad\text{and}\quad B_{N}(\widetilde V,\widetilde W)=N^2\sum_{k=0}^{N-1}B_1(\widetilde V_k,\widetilde W_k).\]
\end{lem}
\begin{proof}
Again, a matter of direct calculations, as follows:
\begin{multline*}
B_0(\widetilde V,\widetilde W)=N^2\sum_{k,j}\int_0^1G\Big(\widetilde V_k'(rN),\widetilde W_j'(rN)\big)\,\mathrm dr\\=
N\sum_{k,j}\int_0^NG\big(\widetilde V_k'(s),\widetilde W_j'(s)\big)\mathrm ds=
N\sum_{k,j,l}\int_l^{l+1}G\big(\widetilde V_k'(s),\widetilde W_j'(s)\big)\mathrm ds\\=
N\sum_{k,j,l}\int_0^{1}G\big(\widetilde V_k'(s+l),\widetilde W_j'(s+l)\big)\mathrm ds=
N\sum_{k,j,l}\omega^{(k-j)l}\int_0^{1}G\big(\widetilde V_k'(s),\widetilde W_j'(s)\big)\mathrm ds\\=N^2
\sum_k\int_0^{1}G\big(\widetilde V_k'(s),\widetilde W_k'(s)\big)\mathrm ds=N^2\sum_kB_0(\widetilde V_k,\widetilde W_k).
\end{multline*}
Similarly,
\begin{multline*}
B_{N}(\widetilde V,\widetilde W)=N^2\int_0^1\sum_{k,j}\Big[G\big(\widetilde V_k'(rN),\widetilde W'_j(rN)\big) +
G\big(\Gamma_{Nr}\widetilde V_k(rN),\widetilde W'_k(rN)\big)\\+G\big(\widetilde V_k'(rN),\Gamma_{Nr}\widetilde W_j(rN)\big)
+G\big(\Gamma_{Nr}\widetilde V_k(rN),\Gamma_{Nr}\widetilde W_j(rN)\big)\\+G\big(\widetilde R_{Nr}\widetilde V_k(rN),\widetilde W_j(rN)\big)\Big]\,\mathrm dr
\\=N^2\sum_{k,j,l}\int_{\frac lN}^{\frac{l+1}N}\Big[G\big(\widetilde V_k'(rN),\widetilde W'_j(rN)\big) +
G\big(\Gamma_{Nr}\widetilde V_k(rN),\widetilde W'_j(rN)\big)\\+G\big(\widetilde V_k'(rN),\Gamma_{Nr}\widetilde W_j(rN)\big)
+G\big(\Gamma_{Nr}\widetilde V_k(rN),\Gamma_{Nr}\widetilde W_j(rN)\big)\\+G\big(\widetilde R_{Nr}\widetilde V_k(rN),\widetilde W_j(rN)\big)\Big]\,\mathrm dr
\\=N\sum_{k,j,l}\int_{l}^{l+1}\Big[G\big(\widetilde V_k'(s),\widetilde W'_j(s)\big) +
G\big(\Gamma_{s}\widetilde V_k(s),\widetilde W'_j(s)\big)\\+G\big(\widetilde V_k'(s),\Gamma_{s}\widetilde W_j(s)\big)
+G\big(\Gamma_{s}\widetilde V_k(s),\Gamma_{s}\widetilde W_j(s)\big)\\+G\big(\widetilde R_{s}\widetilde V_k(s),\widetilde W_j(s)\big)\Big]\,\mathrm ds
\\=N\sum_{k,j,l}\omega^{(k-j)l}\int_{0}^{1}\Big[G\big(\widetilde V_k'(s),\widetilde W'_j(s)\big) +
G\big(\Gamma_{s}\widetilde V_k(s),\widetilde W'_j(s)\big)\\+G\big(\widetilde V_k'(s),\Gamma_{s}\widetilde W_j(s)\big)
+G\big(\Gamma_{s}\widetilde V_k(s),\Gamma_{s}\widetilde W_j(s)\big)\\+G\big(\widetilde R_{s}\widetilde V_k(s),\widetilde W_j(s)\big)\Big]\,\mathrm ds
\\=N^2\sum_k\int_{0}^{1}\Big[G\big(\widetilde V_k'(s),\widetilde W'_k(s)\big) +
G\big(\Gamma_{s}\widetilde V_k(s),\widetilde W'_k(s)\big)\\+G\big(\widetilde V_k'(s),\Gamma_{s}\widetilde W_j(s)\big)
+G\big(\Gamma_{s}\widetilde V_k(s),\Gamma_{s}\widetilde W_k(s)\big)\\+G\big(\widetilde R_{s}\widetilde V_k(s),\widetilde W_k(s)\big)\Big]\,\mathrm ds\\
=N^2\sum_kB_1(\widetilde V_k,\widetilde W_k).\qedhere
\end{multline*}

\end{proof}
\begin{teo}[Fourier theorem]\label{thm:fourier}
\[\mathfrak{sf}\big(\gamma^{(N)}\big)=\sum_{k=0}^{N-1}\lambda_\gamma(\omega^k).\]
\end{teo}
\begin{proof}
This follows immediately from \eqref{eq:lambda1Nsf}, Lemmas~\ref{thm:lem1}, \ref{thm:lem2}, and the following two observations:
\begin{enumerate}
\item the spectral flow of a path of compact perturbations of a fixed symmetry only depends on the endpoints
of the path;
\item the spectral flow is additive by direct sums.\qedhere
\end{enumerate}
\end{proof}
\end{section}
\begin{section}{A superlinear estimate for the spectral flow of an iterate}
We will now use the Fourier theorem and formula \eqref{eq:rellambdagammalambdao} in order to establish estimates on the growth
of the spectral flow for the $N$-th iterate of a closed geodesic $\gamma$. We will use the notations in Subsection~\ref{sub:findimreduction};
an immediate application of Proposition~\ref{thm:formularidotta}, Theorem~\ref{thm:fourier} and Eq. \eqref{spectralflow0} gives the following:
\begin{prop}\label{thm:spflredfindim} Given a closed geodesic and an integer $N\ge1$:
\begin{equation}\label{eq:iterationspfl}
\spfl\big(\gamma^{(N)}\big)=-N\big( \iMaslov(\gamma)\big)-\mathrm n_-(g)+\sum_{k=0}^{N-1}\mathrm{dim}\big(\mathbb J_\gamma^{(1)}(\omega^k)\big)-
\sum_{k=0}^{N-1}\mathrm n_-(b_{\omega^k}),
\end{equation}
where $\omega=e^{2\pi i/N}$.\qed
\end{prop}
Let us estimate the last two terms in formula \eqref{eq:iterationspfl}.
\begin{lem}\label{thm:unifboundeterm}
The quantity $\sum_{k=0}^{N-1}\mathrm{dim}\big(\mathbb J_\gamma^{(1)}(\omega^k)\big)$ is uniformly bounded:
\begin{equation}\label{eq:stimadimJ1}
0\le \sum_{k=0}^{N-1}\mathrm{dim}\big(\mathbb J_\gamma^{(1)}(\omega^k)\big)\le 2\,\mathrm{dim}(M).
\end{equation}
\end{lem}
\begin{proof}
If we identify the space $S=\big\{\widetilde V\in C^2\big([0,1],\C^n\big):\widetilde V\ \text{is solution of \eqref{eq:JacobieqCn}}\big\}$
with  $\C^n\oplus\C^n$ via the map $\widetilde V\mapsto\big(\widetilde V(0),\widetilde V'(0)\big)$, then
for all $z\in\mathds S^1$, $\mathbb J^{(1)}_\gamma(z)$ is identified with a subspace of $\Ker(\mathfrak P_\gamma-z\cdot\mathrm I)$.
The conclusion follows from the fact that $\sum_{z\in\mathds C}\mathrm{dim}\big(\Ker(\mathfrak P_\gamma-z\cdot\mathrm I)\big)\le2\,\mathrm{dim}(M)$.
\end{proof}
A rough estimate for the term containing the index of the Hermitian forms $b_z$ is given in the following:
\begin{lem}\label{thm:crescitaBN}
For all $z\in\mathds S^1$,  $\mathrm n_-(b_z)\le2\,\mathrm{dim}(M)-\mathrm n_0(\gamma)$; if $z$  is not in the spectrum of $\mathfrak P_\gamma$,
then $\mathrm n_-(b_z)\le\mathrm{dim}(M)-\mathrm n_0(\gamma)$. It follows that
\begin{equation}\label{eq:stimasomman-b}
0\le\sum_{k=0}^{N-1}\mathrm n_-(b_{\omega^k})\le N\cdot\big[\mathrm{dim}(M)-\mathrm n_0(\gamma)\big]+4\,\mathrm{dim}(M)^2-2\,\mathrm n_0(\gamma)\,\mathrm{dim}(M).
\end{equation}
\end{lem}
\begin{proof}
Arguing as in the proof of Corollary~\ref{thm:saltoinz=1}, one proves easily that  $\mathrm{dim}\big(\mathbb J_\gamma^{(2)}(z)\big)\le2\,\mathrm{dim}(M)$,
and that $\dim\big(\Ker(b_z)\big)\ge\mathrm n_0(\gamma)$. This last inequality follows from the fact that the space
$\big\{\widetilde W\ \text{solution of \eqref{eq:JacobieqCn}}:\widetilde W(0)=\widetilde W(1)=0\big\}$ is always contained
in the kernel of $b_z$. This proves that $\mathrm n_-(b_z)\le2\,\mathrm{dim}(M)-\mathrm n_0(\gamma)$. When
$z\in\mathds S^1$ is not in the spectrum of $\mathfrak P_\gamma$, then it is shown in the proof of Corollary~\ref{thm:saltoinz=1}
that $\mathrm{dim}\big(\mathbb J_\gamma^{(2)}(z)\big)=\mathrm{dim}(M)$, which gives the improved inequality $\mathrm n_-(b_z)\le\mathrm{dim}(M)-\mathrm n_0(\gamma)$.
Inequality \eqref{eq:stimasomman-b} follows now easily, observing that there are at most $2\,\mathrm{dim}(M)$ eigenvalues of
$\mathfrak P_\gamma$ (on the unit circle).
\end{proof}
\begin{cor}\label{thm:iteratesineq} Set $C_\gamma=4\,\mathrm{dim}(M)^2-2\,\mathrm n_0(\gamma)\,\mathrm{dim}(M)+\mathrm n_-(g)$; then:
\begin{align}\label{eq:superlineargrowth}
\spfl\big(\gamma^{(N)}\big)&\ge\big[-\iMaslov(\gamma)+\mathrm n_0(\gamma)-\dim (M)\big]\cdot N-C_\gamma,\\
\label{eq:superlineargrowth2}
\spfl\big(\gamma^{(N)}\big)&\le-\iMaslov(\gamma)\cdot N-\mathrm n_-(g)+2\dim (M),
\end{align}
for all $N\ge 1$.
\end{cor}
\begin{proof}
Follows easily from \eqref{eq:iterationspfl}, \eqref{eq:stimadimJ1} and \eqref{eq:stimasomman-b}.
\end{proof}
Inequality \eqref{eq:superlineargrowth} becomes interesting when $\iMaslov(\gamma)<\mathrm n_0(\gamma)-\dim (M)$, while \eqref{eq:superlineargrowth2} 
when $\iMaslov(\gamma)>0$. Thus, the question is understanding the asymptotic behavior of $\spfl(\gamma^N)$  when
\begin{equation}\label{eq:boundsmaslov}
\mathrm n_0(\gamma)-\dim (M)\le\iMaslov(\gamma)\le 0;
\end{equation}
note that $\mathrm n_0(\gamma)-\dim (M)\le-1$, and that $C_\gamma$ is bounded uniformly on $\gamma$:
\[2\,\mathrm{dim}(M)^2+2\,\mathrm{dim}(M)+\mathrm n_-(g)\le C_\gamma\le 4\,\mathrm{dim}(M)^2+\mathrm n_-(g).\]
By \eqref{eq:formulaspectralflowperiodic}, \eqref{eq:boundsmaslov} is equivalent to:
\begin{align}\label{eq:boundsspectralflow}
\spfl(\gamma)&\ge\Dim\big(\mathcal  J_\gamma^{\text{per}}\cap\mathcal J_\gamma^0\big)-\mathrm i_{\text{conc}}(\gamma)-\mathrm n_-(g)\\ \spfl(\gamma)&\le
\Dim\big(\mathcal J_\gamma^{\text{per}}\cap\mathcal J_\gamma^0\big)-\mathrm i_{\text{conc}}(\gamma)-\mathrm n_-(g)+\Dim(M)-\mathrm n_0(\gamma)\nonumber.
\end{align}
\begin{lem}\label{thm:subsequence}
   If for some $k\ge1$, $\vert\spfl(\gamma^{(k)})\vert>2\,\Dim(M)+\mathrm n_-(g)$,  then  the sequence $N\mapsto\vert\spfl(\gamma^{(kN)})\vert$ 
   has superlinear growth.
\end{lem}
\begin{proof}
The inequality $\vert\spfl(\gamma^{(k)})\vert>2\,\Dim(M)+\mathrm n_-(g)$ implies that \eqref{eq:boundsspectralflow} is
not satisfied by the iterate $\gamma^{(k)}$; this follows easily considering the trivial inequalities:
\[\Dim\big(\mathcal J_{\gamma^{(k)}}^{\text{per}}\cap\mathcal J_{\gamma^{(k)}}^0\big)\le\Dim\big(\mathcal J_{\gamma^{(k)}}^0\big)\le\Dim(M)-1\]
and
\[\mathrm i_{\text{conc}}\big(\gamma^{(k)}\big)\le\Dim\big(\mathcal J_{\gamma^{(k)}}^\star\big)\le\Dim\big(\mathcal J_{\gamma^{(k)}}\big)=2\,\Dim(M).\]
By Corollary~\ref{thm:iteratesineq}, $N\mapsto\vert\spfl(\gamma^{(kN)})\vert$ has superlinear growth.
\end{proof}
The result of Lemma~\ref{thm:subsequence} is not yet satisfactory; we want to prove that if $\spfl(\gamma^{(N)})$ is not bounded,
then the entire sequence $N\mapsto\vert\spfl(\gamma^{(N)})\vert$ has superlinear growth.
Let us study more precisely the behavior of $\spfl\big(\gamma^{(N)}\big)$ as $N\to\infty$:
\begin{prop}\label{thm:esistenzalimitelambda}
The limit:
\begin{equation}\label{eq;defLgamma}
L_\gamma=\lim_{N\to\infty}\tfrac1N\,\spfl\big(\gamma^{(N)}\big)
\end{equation}
exists, and it is finite.
\end{prop}
\begin{proof}
Using \eqref{eq:iterationspfl}, it suffices to show that the limit:
\[\lim_{N\to\infty}\frac1N\sum_{k=0}^{N-1}\mathrm n_-\big(b_{e^{2\pi ik/N}}\big)\]
exists and is finite.
By Lemma~\ref{thm:J2gammacont}, the map $\mathds S^1\ni z\mapsto\mathrm n_-(b_z)\in\N$ is constant on every connected component
of $\mathds S^1$ that does not contain elements in the spectrum of $\mathfrak P_\gamma$; thus, this function is Riemann
integrable on $\mathds S^1$, and:
\begin{equation}\label{eq:limita1NBN}
0\le\lim_{N\to\infty}\frac1N\sum_{k=0}^{N-1}\mathrm n_-\big(b_{e^{2\pi ik/N}}\big)=
\frac1{2\pi}\int_{\mathds S^1}\mathrm n_-(b_{e^{i\theta}})\,\mathrm d\theta<+\infty.\qedhere
\end{equation}
\end{proof}
Clearly, the following inequality holds:
\[-\iMaslov(\gamma)+\mathrm n_0(\gamma)-\dim (M)\le L_\gamma\le -\iMaslov(\gamma); \]
moreover, in the second inequality, the equality holds if and only if $b_z$ is positive semi-definite at each point of $\mathds S^1$ that does not belong
to the spectrum of $\mathfrak P_\gamma$. With this, we can finally prove the following:
\begin{prop}\label{thm:lineargrowth}
The sequence $\vert\spfl\big(\gamma^{(N)}\big)\vert$ is either bounded or it has superlinear growth.
\end{prop}
\begin{proof}
The thesis is equivalent to proving that $L_\gamma=0$ if and only if the sequence $\spfl\big(\gamma^{(N)}\big)$ is bounded. The ``if'' part is trivial.
Now,  assume by contradiction that $L_\gamma=0$ and that the sequence $\spfl\big(\gamma^{(N)}\big)$ is unbounded. By Lemma \ref{thm:subsequence},
 there exists $k\ge1$ such that the subsequence $N\mapsto\vert\spfl\big(\gamma^{(kN)}\big)\vert$ has superlinear growth, i.e.,
 $\lim\limits_{N\to\infty}\frac1N{\spfl\big(\gamma^{(kN)}\big)}=k L_\gamma\not=0$, which is a contradiction.
\end{proof}

Our goal will now be to prove the occurrence of a \emph{uniform} superlinear growth for the spectral flow of an iterate.
To this aim, we need to study the sequence:
\[\mathcal B_N=\sum_{k=0}^{N-1}\mathrm n_-\big(b_{e^{2\pi i k/N}}\big),\]
which is in a sense, the non trivial part in formula \eqref{eq:iterationspfl}.
A substantial improvement to the result of Lemma~\ref{thm:crescitaBN} can be obtained as follows.
\begin{prop}\label{thm:cresictaBN}
The limit \begin{equation}\label{eq:defKgamma}
K_\gamma=\lim\limits_{N\to\infty}\frac1N\mathcal B_N
\end{equation} 
exists, and it is a nonnegative real number.
This number is zero if and only if $\mathcal B_N$ is bounded, which occurs if and only if
$\mathrm n_-(b_z)$ vanishes almost everywhere on $\mathds S^1$.
If $\mathcal B_N$ is not bounded, then its superlinear growth is uniform in the following sense:
there exist a constant $\alpha\in\R$,  such that for all
$N,P\in\N$:
\begin{equation}\label{firstineq}
K_\gamma P-\alpha\le \mathcal B_{N+P}-\mathcal B_N\le K_\gamma P+\alpha,
\end{equation}
\end{prop}
\begin{proof}
The existence of the limit has already been established in the proof of Proposition~\ref{thm:esistenzalimitelambda}.
From the equality in \eqref{eq:limita1NBN} one obtains easily that the limit is zero if and only if  $\mathrm n_-(b_z)$ vanishes almost everywhere on $\mathds S^1$, and as it has always a finite number of discontinuities (the eigenvalues of $\mathfrak P_\gamma$),
one deduces that this occurs precisely when $\mathcal B_N$ is bounded. As to the last statement, assume that
$e^{i\theta_1},\ldots,e^{i\theta_k}$ are all the eigenvalues of $\mathfrak P_\gamma$ in $\mathds S^1\setminus\{1\}$,
with $0<\theta_1<\ldots<\theta_k<2\pi$; set $\theta_0=0$ and $\theta_{k+1}=2\pi$. For $j=0,\ldots,k$, define
$d_j$ as \[d_j=\lim\limits_{\theta\to0^+}\mathrm n_-\big(b_{e^{i(\theta_j+\theta)}}\big)\ge0;\]
recalling (a) in the proof of Corollary~\ref{thm:saltoinz=1}, $d_j$ is the constant value of the map
$\mathrm n_-(b_z)$ in the arc $\mathcal A_j=\big\{e^{i\theta}:\theta\in\left]\theta_j,\theta_{j+1}\right[\big\}$.
With these notations, we have:
\begin{equation}\label{eq:Kgammasomma}
K_\gamma=\frac1{2\pi}\sum_{j=0}^k {d_{j}(\theta_{j+1}-\theta_{j})};
\end{equation}
by Lemma~\ref{thm:crescitaBN}:
\[d_j\le 2\,\mathrm{dim}(M)-\mathrm n_0(\gamma).\]
By the first part of the proof, the assumption that $\mathcal B_N$ is unbounded is equivalent to the fact that
at least one of the term in the sum \eqref{eq:Kgammasomma} is positive (i.e., $K_\gamma>0$).
Finally, define constants $a_N$ and $C_{N,j}$, for $N\ge1$ and $j\in\{0,1,\ldots,k\}$ by:
\[\begin{aligned}
& a_N=\text{cardinality of}\,\big\{j:N\theta_j\equiv0\mod2\pi\big\},\\
& C_{N,j}=\left\lfloor\frac{N(\theta_{j+1}-\theta_j)}{2\pi}\right\rfloor,
\end{aligned}\]
where $\lfloor\cdot\rfloor$ denotes the integer part function: $\lfloor x\rfloor=\max\{m\in\Z:m\le x\}$.
Clearly, $0\le a_N\le k+1$; moreover, for all $N$ and $j$, the arc $\mathcal A_j$ contains a number of $N$-th roots of unity which is
at most $C_{N,j}+1$ and at least $C_{N,j}-1$. With this in mind, we proceed to the final calculation
giving the desired uniform superlinear growth, as follows:
\begin{multline*} 
\mathcal B_{N+P}-\mathcal B_N=\sum_{l=0}^{N+P-1}\mathrm n_-\big(b_{e^{2\pi i l/(N+P)}}\big)-\sum_{l=0}^{N-1}\mathrm n_-\big(b_{e^{2\pi i l/N}}\big)
\\=\sum_{l=1}^{N+P-1}\mathrm n_-\big(b_{e^{2\pi i l/(N+P)}}\big)-\sum_{l=1}^{N-1}\mathrm n_-\big(b_{e^{2\pi i l/N}}\big)\\\ge
\sum_{j=0}^kd_j(C_{N+P,j}-1)-\sum_{j=0}^kd_j(C_{N,j}+1)-a_N\,\max_{z\in\mathds S^1}\big[\mathrm n_-(b_z)\big]\\
\stackrel{\text{by Lemma~\ref{thm:crescitaBN}}}\ge\sum_{j=0}^kd_j(C_{N+P,j}-C_{N,j}-2)-(k+1)\big[2\,\mathrm{dim}(M)-\mathrm n_0(\gamma)\big]
\\ \ge\sum_{j=0}^k d_{j}(C_{N+P,j}-C_{N,j})-2\Big(\sum_{j=0}^kd_j\Big)-(k+1)\big[2\,\mathrm{dim}(M)-\mathrm n_0(\gamma)\big]
\\ \ge\sum_{j=0}^k\frac{d_{j}(\theta_{j+1}-\theta_{j})}{2\pi}\,P-4\Big(\sum_{j=0}^kd_j\Big)-(k+1)\big[2\,\mathrm{dim}(M)-\mathrm n_0(\gamma)\big]
\\ \stackrel{by\ \eqref{eq:Kgammasomma}}\ge K_\gamma\,P-5(k+1)\big[2\,\mathrm{dim}(M)-\mathrm n_0(\gamma)\big].
\end{multline*}
This concludes the proof of the first inequality in \eqref{firstineq}. The second inequality in \eqref{firstineq} is obtained similarly:
\begin{multline*}
\mathcal B_{N+P}-\mathcal B_N\le
\sum_{j=0}^kd_j(C_{N+P,j}+1)-\sum_{j=0}^kd_j(C_{N,j}-1)+a_N\,\max_{z\in\mathds S^1}\big[\mathrm n_-(b_z)\big]\\
\stackrel{\text{by Lemma~\ref{thm:crescitaBN}}}\le\sum_{j=0}^kd_j(C_{N+P,j}-C_{N,j}+2)+(k+1)\big[2\,\mathrm{dim}(M)-\mathrm n_0(\gamma)\big]
\\ \le\sum_{j=0}^k d_{j}(C_{N+P,j}-C_{N,j})+2\Big(\sum_{j=0}^kd_j\Big)+(k+1)\big[2\,\mathrm{dim}(M)-\mathrm n_0(\gamma)\big]
\\ \le\sum_{j=0}^k\frac{d_{j}(\theta_{j+1}-\theta_{j})}{2\pi}\,P+4\Big(\sum_{j=0}^kd_j\Big)+(k+1)\big[2\,\mathrm{dim}(M)-\mathrm n_0(\gamma)\big]
\\ \stackrel{by\ \eqref{eq:Kgammasomma}}\le K_\gamma\,P+5(k+1)\big[2\,\mathrm{dim}(M)-\mathrm n_0(\gamma)\big].\qedhere
\end{multline*}
\end{proof}
From \eqref{eq:iterationspfl}, \eqref{eq;defLgamma}, and \eqref{eq:defKgamma} one obtains immediately:
\[L_\gamma=-K_\gamma-\iMaslov(\gamma);\]
moreover, we can finally prove the uniform superlinear growth of $\spfl\big(\gamma^{(N)}\big)$:
\begin{prop}\label{thm:uniformgrowth}
With the notations of Corollary \ref{thm:iteratesineq} and Proposition \ref{thm:cresictaBN},
the following inequalities hold:
\begin{equation}\label{eq:uniflingrLgamma}
L_\gamma\cdot P-2\dim (M)-\alpha\le\spfl\big(\gamma^{(N+P)}\big)-\spfl\big(\gamma^{(N)}\big)\le L_\gamma\cdot P+2\dim (M)+\alpha.
\end{equation}
\end{prop}
\begin{proof}
Immediate from Propositions \ref{thm:spflredfindim}, \ref{thm:cresictaBN} and Lemma \ref{thm:unifboundeterm}.
\end{proof}
\begin{cor}\label{thm:boundedorlinear}
The sequence $\spfl\big(\gamma^{(N)}\big)$ is either bounded or it has uniform linear growth.
\end{cor}
\begin{proof}
We have seen in the proof of Proposition \ref{thm:lineargrowth} that the sequence $\spfl\big(\gamma^{(N)}\big)$ is bounded if and only if $L_\gamma=0$, so that the thesis follows from Proposition \ref{thm:uniformgrowth}.
\end{proof}
Denote by $\Lambda M$ the free loop space\footnote{i.e., the Hilbert manifold of all curves $\gamma:\mathds S^1\to M$ having
Sobolev regularity $H^1$.} of $M$, and by $\mathfrak f:\Lambda M\to\R$ the geodesic action functional of $(M,g)$,
whose critical points are well known to be closed geodesics. There is an equivariant action of the orthogonal
group $\mathrm O(2)$ on $\Lambda M$, obtained from the natural action of $\mathrm  O(2)$ on the parameter space $\mathds S^1$.
 A critical $\mathrm O(2)$-orbit of $\mathfrak f$ consists of all closed geodesics that are obtained by rotation
 and inversion of a given closed geodesic in $(M,g)$; it is immediate that all the closed geodesics in the same
 critical orbit have equal spectral flow. 
 Using equivariant Morse theory applied to the
geodesic action functional, Gromoll and Meyer have proved that, in the Riemannian case, the contribution to the homology
 of the free loop space $\Lambda M$ in a fixed dimension $k$ is given only by those closed orbits
 whose Morse index is an integer between $k-\Dim(M)$ and $k$. A key point of their multiplicity result is that, assuming the existence of
 only a finite number of distinct closed prime geodesics, one has a uniformly bounded number of
 distinct orbits with a fixed Morse index (\cite[Corollary~2]{GroMey2}).
Aiming at the development of an equivariant Morse theory for strongly indefinite functionals, we prove an extension of their result, replacing 
the Morse index with the spectral flow.
\begin{prop}\label{thm:geomdistfinnum}
Let $(M,g)$ be a semi-Riemannian manifold that has only a finite number of distinct prime closed geodesics.
Then, for $k\in\Z$ with $\vert k\vert$ sufficiently large, the total number of critical orbits of the geodesic
action functional $\mathfrak f$ in the free loop space $\Lambda M$ having spectral flow equal to $k$ is  bounded uniformly in $k$.
\end{prop}
\begin{proof}
Let $\gamma_1$,\ldots,$\gamma_r$ be the family of all distinct prime closed geodesics in $M$
and fix some integer $k$ with $\vert k\vert>2\,\Dim(M)+\mathrm n_-(g)$ (recall Lemma~\ref{thm:subsequence}).
We can remove from the family those geodesics whose spectral flow is not unbounded by iteration, and assume that
all these geodesics have iterates with unbounded spectral flow; in particular, $L_{\gamma_i}\ne0$ for all $i$.
For $i=1,\ldots,r$, let $N_i\ge 1$ be the first integer such that $\spfl\big(\gamma_i^{(N_i)}\big)=k$; if no such integer
$N_i$ exists, we can remove also  $\gamma_i$ from the family. From \eqref{eq:uniflingrLgamma} we obtain easily
that $\vert\spfl\big(\gamma^{(N_i+P)}\big)-k\vert>1$ when 
\[P>\frac{1+\alpha+2\,\Dim(M)}{\vert L_{\gamma_i}\vert}.\]
Thus, there are at most:
\[r+\sum_{i=1}^r\left\lfloor\frac{1+\alpha+2\,\Dim(M)}{\vert L_{\gamma_i}\vert}\right\rfloor\]
critical orbits of $\mathfrak f$ with spectral flow equal to $k$.
\end{proof}

\smallskip

Finally, let us recall that $\gamma$ is said to be \emph{hyperbolic} if the linearized Poincar\'e map $\mathfrak P_\gamma$ does not
have eigenvalues on the unit circle. For hyperbolic geodesics, the iteration formula for the spectral flow has a simple expression:
\begin{prop}\label{thm:gammahyper}
If $\gamma$ is hyperbolic, then:
\[\spfl\big(\gamma^{(N)}\big)=N\spfl(\gamma)+(N-1)\,\mathrm n_-(g).\]
\end{prop}
\begin{proof}
This follows easily from Corollary~\ref{thm:saltoinz=1} and Theorem~\ref{thm:fourier}.
\end{proof}
\end{section}

\end{document}